\documentclass{amsart}
\usepackage[utf8]{inputenc} 
\usepackage[T1]{fontenc}
\usepackage{lmodern}
\usepackage{graphicx}
\usepackage[english]{babel}
\usepackage{wrapfig}
\usepackage{amsthm}
\usepackage{setspace}
\usepackage{amsmath}

\usepackage{amssymb}

\usepackage{comment}
\usepackage{caption}
\usepackage[a4paper]{geometry}  

\usepackage[bottom]{footmisc}
\usepackage[all]{xy}
\usepackage{float} 
\usepackage{hyperref}
\usepackage{setspace}
\newsavebox{\auteurbm}
  {\small\slshape%
  \savebox{\auteurbm}{\upshape\sffamily#1}%
  \begin{flushleft}}
  {\\[4pt]\usebox{\auteurbm}
  \end{flushleft}\normalsize\upshape}
  
  \usepackage{mathtools}
  \usepackage{dsfont}
  \usepackage{amsmath}
  \usepackage{amsthm}

  \DeclareMathOperator*{\Li}{\mathcal{L}}

  \theoremstyle{definition}
  \newtheorem{defn}{Definition}[section]
  
  \newtheorem*{defn*}{Definition}
 
  \theoremstyle{plain}
  \newtheorem{thm}{Theorem}[section]
  
  \newtheorem{prop}{Proposition}[section]
  \newtheorem*{prop*}{Proposition}
  \newtheorem{lem}{Lemma}[section]
  \newtheorem{cor}{Corollary}[section]
   \newtheorem*{cor*}{Corollary}
  \newtheorem*{theo*}{Theorem}
  \newtheorem*{thm*}{Theorem}
  
  \theoremstyle{remark}
  
  \newtheorem{rem}{Remark}[section]
 
  \newtheorem{ex}{Example}[section]
 
  \newtheorem{nota}{Notation}[section]

 \usepackage[toc,page]{appendix}
\usepackage{calrsfs}
\usepackage{tikz-cd}

\newcommand{\R}{\mathbb{R}}
\newcommand{\A}{\mathcal{A}}
\newcommand{\B}{\mathcal{B}}

\DeclareMathOperator*{\F}{U}

\newcommand{\M}{\mathcal{M}}

\newcommand{\U}{\operatorname{\mathcal{U}}}

\usepackage{pstricks,pst-node,pst-tree}

\usepackage{bm}

\usepackage{relsize}

\DeclareMathOperator{\Supp}{\text{Supp}}

 \usepackage{xcolor}

 \newcommand\N{\mathbb{N}}
 \newcommand\cat[1]{\textbf{#1}}

\newcommand{\Pa}{\mathcal{P}}

\newcommand{\pr}{\text{pr}}

\DeclareMathOperator*{\vect}{\textbf{Gr}}
\DeclareMathOperator*{\W}{U}

\newcommand{\colim}{\operatorname{colim}}

\usepackage{appendix}

\title{Interaction decomposition for presheaves}

\author{Grégoire Sergeant-Perthuis}

\usepackage{pdfpages} 

\usepackage[intoc]{nomencl}
\makenomenclature

\usepackage{pdfpages}

\newcommand\Pf{\hat{\mathcal{P}}_f}
\newcommand{\Mod}{\cat{Mod}}
\newcommand{\Split}{\cat{Split}}
\def\Plus{\small\texttt{+}}

\DeclareMathOperator{\core}{core}

\newcommand{\Z}{\mathbb{Z}}

\newcommand{\im}{\operatorname{im}}

\newcommand{\id}{\operatorname{id}}

\makeatletter
\newtheorem*{rep@theorem}{\rep@title}
\newcommand{\newreptheorem}[2]{%
\newenvironment{rep#1}[1]{%
 \def\rep@title{#2 \ref{##1}}%
 \begin{rep@theorem}}%
 {\end{rep@theorem}}}
\makeatother

  \theoremstyle{plain}
\newreptheorem{thm}{Theorem}
\newreptheorem{prop}{Proposition}
\newreptheorem{cor}{Corollary}

\makeatletter
\newcommand{\mylabel}[2]{#2\def\@currentlabel{#2}\label{#1}}
\makeatother

\numberwithin{equation}{section}

\newcommand{\Vect}{\cat{Vect}}

\usepackage{multiaudience}
\SetNewAudience{long}

\usepackage{minitoc}  
 
\usepackage{changepage} 
 
\usepackage{chngcntr}
\usepackage{etoolbox}
 \makeatletter
   {\par}
\makeatother

\AtBeginEnvironment{subappendices}{%
\section*{Appendix}
\addcontentsline{toc}{section}{Appendices}
\counterwithin{figure}{section}
\counterwithin{table}{section}
}

\usepackage{cite}

\begin{document}

\begin{abstract}
Consider a collection of vector subspaces of a given vector space and a collection of projectors on these vector spaces, can we decompose the vector space into a sum of vector subspaces such that the projectors are isomorphic to projections? We provide an answer to this question by extending the relation between the intersection property and the interaction decomposition introduced in \cite{GS2} to the projective case. This enables us to classify the decompositions into interaction subspaces for factor spaces when the projectors are given by conditional expectations of a probability measure. We then extend these results for presheaves from a poset to the category of modules by adding the data of a section functor when it exists.

\noindent \textbf{MSC2020 subject classifications:} Primary 06F25; secondary 62H22 \\
\noindent \textbf{Keywords.} Decomposition into interaction subspaces, Ordered vector spaces, Chaos decomposition, 

\end{abstract}

 \maketitle

\section{Introduction}
\subsection{Motivation}

The decomposition into interaction subspaces is an important (class of) construction in probability theory studied for example in the context of graphical models, for the study of the marginal problem \cite{Kellerer1964,Lauritzen,Speed,Yeung,Matus,Davidson,Haberman}, but also in statistical physics where the chaos decomposition given by the Hermite-Ito polynomials \cite{Sinai} shares the same spirit. Let us recall one of these decompositions; let $I$ and $E=\prod_{i\in I}E_i$ be finite sets, for $a\subseteq I$, the $a-$factor subspaces (or $a-$factor space) $\F(a)$ is the space of functions that factor through the projection $p_a:E\to \prod_{i\in a} E_i$. There is a collection of vector subspaces, called the interaction subspaces, $(S_a,a\in \A)$ such that for any $a\in \A$,

\begin{equation}\label{equation-orthogonal-decomposition}
\F(a)= \bigoplus_{b\leq a} S_b
\end{equation}

where the $(S_a,a\in \A)$ are two by two orthogonal, with respect to the canonical scalar product on $\R^E$. There are many other decompositions if one does not ask the $S_a,a\in \A$ to be orthogonal. In fact in \cite{GS1}\cite{GS2} we explain when and how to build such decompositions for functors from a poset to the category of vector spaces.

One way we choose to distinguish the different decompositions is to specify the collection of projectors $\pi_a: \R^E\to \R^E$ on the $a-$factor spaces. In the example we consider those are the orthogonal projectors and in Lauritzen's proof of the decomposition into interaction subspaces (Appendix B of his reference book Graphical Models \cite{Lauritzen}) it is stated that if $(s_a,a\in \A)$ is such that for any $a\in \A$,

\begin{equation}\label{chapitre-3-introduction-decomposition}
\pi_a= \sum_{b\leq a} s_b
\end{equation}

then the $(s_a,a\in \A)$ are the orthogonal projectors on the $(S_a,a\in \A)$. In fact there can be only one collection that satisfies Equation \ref{chapitre-3-introduction-decomposition}.

More generally we will call any collection of projectors $(\pi_a,a\in \A)$, of a vector space $V$, that is such that the collection $(s_a,a\in \A)$ defined by Equation \ref{chapitre-3-introduction-decomposition} is such that,

\begin{equation}\label{chapitre-3-introduction-decomposition-1}
s_as_b= 1[a=b]s_a
\end{equation}

decomposable and the collection $(s_a,a\in \A)$ will be its decomposition, that we will also call its interaction decomposition in honor to the decomposition into interaction subspaces for factor spaces. Not all collections of projectors are decomposable; in this article we give an intrinsic necessary and sufficient condition on the collection of projectors for them to be decomposable which is that, when $\A$ is a meet semi-lattice, for any $a,b\in \A$ such that $b\leq a$, 

\begin{equation}\tag{I}
\pi_a\pi_b= \pi_{a\cap b}
\end{equation}

There exists several examples of interaction decompositions for factor spaces or decompositions into interaction subspaces in the literature. We show that if one considers collection of projectors given by conditional expectations with respect to the $a-$factor spaces and to a given probability measure, we can characterize all of them. More precisely for any probability measure $\mathbb{P}\in \mathbb{P}(E)$, one can define a semidefinite symmetric bilinear form on $\R^E$; conditional expectations with respect to the $a-$factor spaces are analogous to orthogonal projections on these subspaces. We show that these projectors are decomposable if and only if the probability measure is a product measure.

\begin{nota}
For any poset, let $\U(\A)$ denote the set of lower-sets of $\A$, i.e subsets of $\A$ such that for any $a\in \A$,$b\in\B$ such that $a\leq b$ one has that $a\in \B$. We shall also call $\U(\A)$ the poset topology of $\A$. For $a\in \A$, let $\hat{a}=\{b\in \A : b\leq a\}$, $\hat{a}\in \U(\A)$.
\end{nota}

\subsection{Main results}

\subsubsection{Decomposable presheaf}

The collection of projectors we have in mind when dealing with probability with a categorical flavour can be encoded by a presheaf, $F$, from a poset $\A$ to the category of vector spaces $(k)-\Vect$ with respect to functor $G:\A\to \Vect$, with $k$ a commutative field, such that $F$ is a retraction of $G$. In most generality we consider couples of functor/ presheaf, $(G,F)$, from a poset $\A$ to the category of $R-$modules, where $R$ is a commutative ring.

\begin{defn*}[Category $\Split$]
Let $\cat{C}$ be any category, let $\cat{Split}(\cat{C})$ be the subcategory of $\cat{C}\times \cat{C}^{op}$ that has as objects $(M,M)$ with $M$ an object of $\cat{C}$ and for any two objects of $\cat{C}$,$M,M_1$, the morphism of $\Split$ are given by

\begin{equation}
 \cat{Split}(\cat{C})\left((M,M),(M_1,M_1)\right)=\{(s,r)\in \Mod(M,M_1)\times \Mod(M_1,M) : rs=\id\}
\end{equation}
\end{defn*}

A functor from a poset $\A$ to $\Split(\cat{C})$ can be encoded by a couple of functor/presheaf, $(G,F)$, from $\A$ to $\cat{C}$, such that for any $a,b\in \A$ such that $b\leq a$,

\begin{equation}
F^a_bG^b_a= \id
\end{equation}

We will consider two categories, the category of modules $\Mod$ and its subcategory for which all morphisms are isomorphisms. For a collection of functors $(G_a:\A\to\Mod,a\in \A)$ from $\A$ to $\Mod$ and a collection of presheaves $(F_a:\A \to \Mod,a\in \A)$ we will denote $\prod_a G_a 1[a\leq .]$ and respectively $\prod_a F_a1[a\leq .]$ the functors and presheaves defined as follows, for any $a\in \A$,

\begin{eqnarray}
\prod_b G_b 1[b\leq .](a)= \prod_{b\leq a} G_b(a)\\
\prod_b F_b 1[b\leq .](a)=\prod_{b\leq a} F_b(a)
\end{eqnarray}

and the morphisms are respectively the associated inclusions and projections, in other words for $b,d \in \A$ such that $b\leq a$, for $v\in \bigoplus_{c\leq b} G_c(b)$ and $w\in \bigoplus_{c\leq a} G_c(a)$ then,

\begin{eqnarray}
\prod_c G_c 1[c\leq .]^b_a(v)(d)={G_c}^b_a(v_c)1[d\leq b] \\
\prod_c F_c 1[c\leq .]^a_b(w)(d)={F_c}^a_b(w_c) 1[d\leq b]
\end{eqnarray}

\begin{defn}[Core of a category]
The core of a category $\cat{C}$, $\core\cat{C}$, is the subcategory of $\cat{C}$ that has the same objects than $\cat{C}$ but whose morphisms are the isomorphisms of $\cat{C}$.
\end{defn}

We will note $\Split(\Mod)$ simply as $\Split$.

\begin{defn*}[Decomposable functors to $\Split$]
Let $H$ be a functor from $\A$ to $\Split$. $H$ is decomposable when there is a collection $((G_a,F^a),a\in \A)$ of functors from $\A$ to $\Split(\core \Mod)$ such that,

\begin{equation}
H\cong (\prod_c G_c 1[c\leq .], \prod_c F_c 1[c\leq .])
\end{equation}

When $H$ is decomposable we shall call $(\prod_{a\in \A}G_a1[a\leq .],\prod_{a\in \A}F^a1[a\leq .])$ its decomposition and note it as $(\prod_{a\in \A}S_a,\prod_{a\in \A}S^a)$.
\end{defn*}

\subsection{Intersection property for functors from $\A$ to $\Split$}

In order to state the intersection property we must consider posets that are such that $\hat{a}$ is finite for any $a\in \A$; we will note the class of these posets as $\Pf$.\\

\begin{defn*}[Intersection property]
Let $(G,F)$ be a functor from $\A\in \Pf$ to $\Split$. For any $\alpha,a\in \A$ such that $a\leq \alpha$, let $\pi^{\alpha\alpha}_{\alpha a}=G^a_\alpha F^\alpha_a$, $\pi^{\alpha\alpha}_{\alpha a}$ is a projector. For a given $\alpha$, we shall denote this collection as $\pi^{\alpha}$. Let for any $\alpha \in \A$, $(s^{\alpha}_a,a\leq \alpha)$ be characterized by,

\begin{equation}
\pi^{\alpha}_a= \sum_{b\leq a} s^\alpha_a 
\end{equation}

$(G,F)$ is said to satisfy the intersection property for any $\alpha \in \A$ and any $a,b\leq \alpha$,

\begin{equation}\tag{I}
\pi^{\alpha}_a\pi^{\alpha}_b= \sum_{\substack{c: c\leq a\\ c\leq b}} s_c
\end{equation}

\end{defn*}

When the poset $\A$ is meet semi-lattice for any two elements $a,b\in \A$ there is a meet $a\wedge b$; we will note this meet as $a\cap b$. In this case, $(I)$ becomes, for any $\alpha \in \A$ and any $a,b\leq \alpha$,

\begin{equation}
\pi^{\alpha}_a\pi^{\alpha}_b=\pi^{\alpha}_{a\cap b}
\end{equation}

\subsubsection{Main theorem}

\begin{repthm}{decomposable-intersection-presheaf}
Let $(G,F)$ be a functor from $\A\in \Pf$ to $\Split$, $(G,F)$ satisfies the intersection property if and only if $(G,F)$ is decomposable. 

\end{repthm}

A corollary of this equivalence theorem is that one can extend know results on the decomposition into interaction spaces (see \cite{Kellerer1964,Lauritzen,Speed,Yeung,Matus}) for factor spaces by characterizing the probability distributions for which one can build a decomposition.

\begin{cor*}[Interaction Decomposition for factor spaces]
Let $I$ be a finite set, $(E_i,i\in I)$ a collection of finite sets, and $\mathbb{P}$ a probability measure on $E$, $(E_a[.| \mathcal{F}_a],a\in \mathcal{P}(I))$ is decomposable if and only if $\mathbb{P}$ is a product measure, i.e there is $(p_i\in \mathbb{P}(E_i),i\in I)$ such that $\mathbb{P}= \underset{i\in I}{\otimes}p_i$.\\

\end{cor*}

\subsection{Structure of this document}

This document is composed of mainly two parts. We first motivate the framework for decomposability in the context of  collections of projectors and prove the equivalence theorem between intersection property and decomposability in this context (Section \ref{chapitre-3-section-1} \ref{chapitre-3-section-2}). This allows us to characterize the probability distribution that induces, by conditioning, a decomposable collection of projectors. \\

In the second part of this document we extend these results to functors $(G,F)$ from $\A\in \Pf$ to $\Split$ (Section \ref{chapitre-3-section-3}). In order to state the intersection property we show that, by nesting in each element $\alpha\in \A$ the poset $\hat{\alpha}$ and by bringing back $F|_{\hat{\alpha}}$ in $F(a)$, we can use the definitions and results for collection of projectors in $F(\alpha)$ to extend them to functors to $\Split$. We then show that from the  decomposition that can be build for each $\alpha$ we can extract a decomposition for $(G,F)$ proving therefore the main theorem of this document.

\section{Decomposability for collections of projectors}\label{chapitre-3-section-1}
\subsection{Projections on factor spaces}
As the decomposition into interaction subspaces is stated for collection of random variables, let us first give some reasons for why solving this question can be interesting in probability; let $I$ be a finite set, it indexes random variables, let for all $i\in I$, $E_i$ be a finite set with the discrete $\sigma$-algebra, in which the $i$-th random variable takes its values, $E=\prod_{i\in I}E_i$ is the configuration space; for $\omega\in E$, one has that $p_i(\omega)=\omega(i)$, and for $a\subseteq I$ non empty, we will note $\omega_{|a}$ as $x_a$. We will call $E_a=\prod_{i\in a} E_i$ and,

$$\begin{array}{ccccc}
p_a& : &E & \to & E_a\\
& & x & \mapsto &x_a\\
\end{array}$$

Let $\ast$ be a given singleton. Then there is only one application of domain $E$ to $\ast$ that we call $\pi_\emptyset$; we pose $\omega_\emptyset=p_\emptyset (\omega)$. The $\sigma$-algebra one considers implicitly on $\prod_{i\in I}E_i$ will be the Borel algebra with respect to the product topology, i.e. the smallest algebra that makes the projections, for any $a\subseteq I$, $\pi_a$ measurable, here as the $E_i$ are finite and $I$ is finite it coincides with the discrete $\sigma$-algebra on E. Let us denote $\mathcal{F}_a$ the smallest $\sigma$-algebra that makes $p_a$ measurable, i.e generated by the cylindrical events $\{\omega_a\}:=\{\omega_1 \in E \ : {\omega_1}_a=\omega_a\}$.\\

\begin{nota}
For any measurable space $(E,\mathcal{F})$ let us denote $\mathcal{M}(E,\mathcal{F})$ the set of measurable function and $\mathcal{M}_b(E,\mathcal{F})$ the set of bounded measurable functions, i.e. $\mathcal{L}^{\infty}(E,\mathcal{F})$. Let $\mathbb{P}(E)$ be the set of probability measures of $E$.

\end{nota}

\begin{nota}
We shall note the image measure ${p_a}_{*}\mathbb{P}$, i.e. the marginalisation of $\mathbb{P}$ over $E_{\overline{a}}$, as $\mathbb{P}_a$.
\end{nota}

\begin{defn}\label{Factor-subspaces}

For any $a\in \Pa(I)$, let $V(a)$ be the vector subspace of $V$ constituted of functions $f$ that can be factorised by $p_a$, in other words there is $\tilde{f}$ such that $f=\tilde{f}\circ p_a$. $V(a)$ is called the $a$-factor space, $V(a)=\mathcal{M}_b(E,\mathcal{F}_a)$.

\end{defn}

As $E$ is finite one can define for a probability $\mathbb{P}$ its support,

$$\Supp \mathbb{P}=\{\omega\in E\quad : \mathbb{P}(\omega)\neq 0\}$$

Let $\mathbb{P}$ be the probability measure on $E$ associated to the collection of random variables $I$, for any $a\subseteq I$, let $\pi_a:V\to V(a)$ be  such that $\pi_a= E_a[\quad| \mathcal{F}_a]$ the conditional expectation with respect to the $a$ factor space; we take as convention that for any cylinder events if $\mathbb{P}(\{\omega_a\})=0$, $E_a[\quad | \mathcal{F}_a](\omega)=0$; therefore for any $a\in \A$ and $f\in \mathcal{M}_b(E)$, 

\begin{equation}
E[f|\mathcal{F}_a]=1[\in \Supp \mathbb{P}_a]\sum_{\omega^{'}:\omega^{'}_a=\omega_a}\frac{\mathbb{P}(\omega^{'})}{\mathbb{P}_a(\omega_a)}f(\omega^{'})
\end{equation}

\subsection{Collection of projectors and functoriality}

\begin{nota}
 For a module $M$ we will denote $\vect M$ the poset of sub-modules of a module $M$. For two submodules $M_1,M_2\in \vect M$ such that $M_1\subseteq M_2$, we will denote $i^{M_1}_{M_2}$ the inclusion of $M_1$ into $M_2$. $\hom(\A,\vect M)$ will be the set of increasing functions between a poset $\A$ and $\vect M$; when $M$ is a vector space, $\vect M$ is also called the Grassmannian of $M$. A poset is a particular category with at most one morphism between two object; any increasing function $\F\in \hom(\A,\vect M)$ is in particular a functor.
 \end{nota}
 
\begin{nota}
We shall note the set of endomorphisms of $V$ as $\Li(V)$. For $\pi=(\pi_i\in \Li(V),i\in I)$ the collection $(\im \pi_i,i\in I)$ will be called the image of $\pi$ and denoted as $\im \pi$.
\end{nota}

\begin{prop}\label{image-functor-condition}
Let $\pi$ be a collection of projectors of $V$ over $\A$. $\im \pi\in \hom(\A,\vect{V})$ if and only if for any $a,b\in \A$ such that $b\leq a$, $\pi_a\pi_b=\pi_b$: 
\end{prop}

\begin{proof}

Assume that $\im \pi_b\subseteq \im \pi_a$ for $b\leq a$ then $\pi_a\pi_b=\pi_b$; suppose $\pi_a\pi_b=\pi_b$ then $\im \pi_b \subseteq \im \pi_a$.
\end{proof}

\begin{rem}
When $\im \pi \in \hom(\A,\vect{V})$, let for $b\leq a$, $G(a)= \im \pi_a$ and $G^b_a=i^{\im \pi_b}_{\im \pi_a}$; $G$ is a functor from $\A$ to $\vect{V}$, where $\vect{V}$ has as morphisms inclusions; we shall call $G(\pi)$ the canonical functor associated to $\pi$.
\end{rem}

\begin{nota}
For $V$ a vector space, we shall note $\cat{sub}(V)$ the subcategory of subvector spaces of $V$; morphisms are any linear transformations from one subspace to another.

\end{nota}

\begin{defn}\label{inverse-system-homomorph}
Let $\A$ be a poset, $\pi=(\pi_a,a\in \A)$ a collection of endomorphisms of $V$. If for any $a,b\in\A$ such that $b\leq a$ there is $f^a_b\in \Li(V)$ such that, 

\begin{equation}\label{eq-inverse-system-homomorph}
\pi_b=f^a_b\circ \pi_a
\end{equation}

we will say that $\pi$ is presheafable in $\cat{sub}(V)$. 
\end{defn}

\begin{defn}
For any poset $\A$, we shall call the nerve of $\A$, denoted as $N(\A)$, the set of strictly increasing sequence of $\A$; in particular we shall call $N(\A)_n$ the set of strictly increasing sequence of $n$-elements.

\end{defn}

\begin{ex}
$N(\A)_2=\{(a,b)\in \A^2  : a<b\}$.

\end{ex}

\begin{nota}
Let $X,A,Y,B$ be four sets such that $A\subseteq X$ and $B\subseteq Y$. Let $f:X\to Y$ be a function such that $f(A)\subseteq B$, we shall note $f|^B_A$ is restriction on $A$ and $B$.
\end{nota}

\begin{prop}\label{inversable-1}
Let $\pi$ be presheafable in $\cat{sub}(V)$ and $(f^b_a, (a,b)\in N(\A)_1)$ a collection that satisfies Equation (\ref{eq-inverse-system-homomorph}); let $F(\pi)^a_a=id$ and $F(\pi)^a_b={f^a_b}|^{\im \pi_a}_{\im \pi_b}$ when $b\leq a$, $F(\pi)$ is a presheaf from $\A$ to $\cat{sub}(V)$ that we shall call the canonical presheaf associated to $\pi$. 

\end{prop}

\begin{proof}
For any $(a,b)\in N(\A)_1$, $\im f^a_b(\im \pi_a)\subseteq \im \pi_b$ therefore $F^b_a$ is well defined. Furthermore for $a\geq b\geq c$ and $v\in V$, $F^b_cF^a_b\pi_a(v)=f^b_c\pi_b(v)=\pi_c(v)=F^a_c\pi_a(v)$; as $\pi_a$ is surjective on its image, $F^b_cF^a_b=F^a_c$.\\

Let $({f_1}^b_a,(a,b)\in N(\A)_1)$ that satisfies Equation (\ref{eq-inverse-system-homomorph}), for $b\leq a$, $f^a_b\pi_a=\pi_b={f_1}^a_b\pi_b$ therefore $F^a_b={F_1}^a_b$ which justifies why we can call $F$ canonical.

\end{proof}


\begin{prop}\label{presheaf-projector}
Let $\A$ be a poset, $\pi=(\pi_a,a\in \A)$ a collection of projectors of $V$. $\pi$ is presheafable in $\cat{sub}(V)$ if an only if for any for any $a,b\in N(\A)_1$, $\pi_a\pi_b=\pi_a$. 

\end{prop}

\begin{proof}
Let us assume that $\pi$ is presheafable, and let $(f^b_a, (a,b)\in N(\A)_1)$ be such that Equation (\ref{eq-inverse-system-homomorph}) holds for any $(a,b)\in N(\A)_1$. Let $(a,b)\in N(\A)_1$ and $v\in \im \pi_b$ then, $\pi_a\pi_b=f^b_a\pi_b \pi_b=f^b_a\pi_b=\pi_a$; let $\pi$ be such that for $(a,b)\in N(\A)_1$, $\pi_a\pi_b=\pi_a$ then by definition $\pi$ is presheafable ($f^b_a=\pi_a$).
\end{proof}


\begin{rem}
A presheafable collection of projectors over $\A$, doesn't generally satisfy $\im \pi \in \hom(\A,\vect{V})$ and a collection of projectors such that $\im \pi \in \hom(\A,\vect{V})$ is not necessarily presheafable: let $\pi_1,\pi_2$ two projectors, $\pi_1\pi_2=\pi_2$ is equivalent to having $\im \pi_2\subseteq \im \pi_1$, as one can rewrite $\pi_1\pi_2=\pi_2$ as $(id-\pi_1)\pi_2=0$;  $\pi_1\pi_2=\pi_1$ is equivalent to  $\ker \pi_2\subseteq \ker \pi_1$ as one can rewrite $\pi_1\pi_2=\pi_1$ as $\pi_1(id-\pi_2)=0$ which says that $\im(id-\pi_2)\subseteq \ker \pi_1$ and $\im(id-\pi_2)=\ker \pi_2$.
\end{rem}

\begin{prop}
Let $I$ be a finite set, $E=\prod_{i\in I}E_i$, $\mathbb{P}$ a probability measure on $E$; $\left(\pi_a=E[ \quad |\mathcal{F}_a],a\in \mathcal{P}(I)\right)$ is such that $\im\pi \in \hom(\A,\vect{V})$ and is presheafable.
\end{prop}
\begin{proof}
 For any $f\in \M_b(E)$ and $a,b\in \A$ such that $b\leq a$, $E[E[f|\mathcal{F}_a]|\mathcal{F}_b]=E[f|\mathcal{F}_b]$ therefore by Proposition \ref{presheaf-projector} $\pi$ is presheafable; for $f\in \mathcal{F}_b$, $E[f|\mathcal{F}_a]=f1[.\in \Supp\mathbb{P}_b]$, $E[E[f|\mathcal{F}_b]|\mathcal{F}_a]= E[f|\mathcal{F}_b]1[.\in \Supp \mathbb{P}_b]=  E[f|\mathcal{F}_b]$, therefore by Proposition \ref{image-functor-condition}.
\end{proof}

\subsection{Decomposable collections of projectors implies decomposable collection of vector spaces}
\begin{nota}
Let $\A$ be a poset, let $\A^{\Plus}=\A\oplus 1$ be the poset sum of $\A$ and the one element poset $1$, i.e. any $a\in \A^{\Plus}$ is in $\A$ or is equal to $1$ and  $1$ is a final element. 
\end{nota}

\begin{defn}\label{decomposable}
Let $V$ be a vector space, $\A$ be a finite poset and $(\pi_a,a\in \A)$ a collection of endomorphisms of $V$. Let $\F(\pi)$ be the application from $\A^{\Plus}$ to $\vect{V}$ such that for any $a\in \A$, $\F(a)=\im \pi(a)$, $\F(1)=V$. We say that $(\pi_a,a\in \A)$ is decomposable if and only if there is a collection of vector subspaces of $V$, $(S_a,a\in \A^{\Plus})$ such that

\begin{enumerate}

\item $\bigoplus_{a\in\A^{\Plus}}i^{S_a}_V: \bigoplus_{a\in\A^{\Plus}}S_a\to V$ is an isomorphism; let us recall that for any $w=(w_a,a\in \A^{\Plus})\in \bigoplus_{a\in\A^{\Plus}}S_a$, $\bigoplus_{a\in\A^{\Plus}}i^{S_a}_V(w)= \sum_{a\in\A^{\Plus}}w_a$.

\item for any $a\in \A$, $\pi_a(\sum_{b\in \A^{\Plus}} w_b)=\sum_{b\leq a} w_b$. 
\end{enumerate}

We shall note $\bigoplus_{a\in\A^{\Plus}}i^{S_a}_V$ as $\phi$ and the projections on the $S_a$ as $s_a$, i.e.  for any $a\in \A^{\Plus}$, $s_a:V\to V$ is such that $s_a(v)=w_a$; we shall say that $(s_a,a\in \A)$ is the decomposition of $\pi$.

\end{defn}

Let us recall what a decomposable collection of vector spaces is. 

\begin{defn}\label{def-decompositio-1}
$(S_a, a\in \A)$ is a decomposition of $\W \in \hom(\A,\vect(V))$ if and only if, 

\begin{enumerate}
\item for all $a\in \A$, $S_a \in \vect(\F(a))$.

\item for all $a\in \A$, $\bigoplus_{b\leq a} i^{S_b}_{U(a)}:\bigoplus_{b\in \hat{a}} S_b \to \F(a)$ are isomorphisms and $\bigoplus_{a\in \A}i^{S_a}_V$ is an isomorphism on its image.

\end{enumerate}

We shall note $i^{S_b}_{U(a)}$ as $\phi_a$; we shall say that $\F$ is decomposable.
\end{defn}

\begin{prop}
Let $\pi=(\pi_a,a\in \A)$ be decomposable and $(S_a,a\in \A^{\Plus})$ the decomposition of $\pi$, let for any $a\in \A^{\Plus}$, $\phi_a: \bigoplus_{b\leq a}S_a\to \F(\pi)(a)$ be such that $\phi_a(\bigoplus_{b\leq a}w_b)=\sum_{b\leq a}w_b$. Then $(\phi_a,a\in \A)$ is a natural transformation from $\bigoplus S^a$ to $G(\pi)$ and a natural transformation from $\bigoplus S_a$ to $F(\pi)$, it is also an isomorphism, $\F(\pi)$ is decomposable and $(\pi_a,a\in \A)$ is presheafable.

\end{prop}

\begin{proof}

For any $v\in \underset{b\in \A^{\Plus}}{\bigoplus}S_b$, and $a\in \A$, $\phi^{-1}\pi_a(\phi(v))=i^{\underset{b\leq a}{\bigoplus}S_b}_{\underset{b\in \A^{\Plus}}{\bigoplus}S_b}pr^{\underset{b\in \A^{ \Plus }}{\bigoplus}S_b}_{\underset{b\leq a}{\bigoplus}S_b}(v)$ therefore $\phi(i^{\underset{b\leq a}{\bigoplus}S_b}_{\underset{b\in \A^{\Plus}}{\bigoplus}S_b}\underset{b\leq a}{\bigoplus}S_b)=\F(a)$ and $\phi_a=\phi i^{\underset{b\leq a}{\bigoplus}S_b}_{\underset{b\in \A^{\Plus}}{\bigoplus}S_b}|^{\F(a)}$ is well defined and is an isomorphism; from this remark one can conclude that $(\phi_a,a\in \A)$ is a natural transformation from $\bigoplus S^a$ to $G(\pi)$ and a natural transformation from $\bigoplus S_a$ to $F(\pi)$, it is also an isomorphism, $\F(\pi)$ is decomposable. $(\pi_a,a\in \A)$ is presheafable because $\pr^{\underset{b\in \A^{\Plus}}{\bigoplus}S_b}_{\underset{b\leq a}{\bigoplus}S_b}$ is presheafable.\\

\end{proof}


\begin{defn}[Meet semi-lattice]
Let $\A$ be a poset, $a,b\in\A$. We shall say that $\A$ has a meet for $(a,b)$ when there is $d$ such that,

$$\forall c\in \A,\quad  c\leq a \quad \& \quad c\leq b \implies c\leq d$$

$d$ is unique and we shall note it $a\cap b$. We will call meet semi-lattice any poset that has all meets for any couple.

\end{defn}

\begin{proof}
Let $d$ and $d_1$ be two intersections for $a,b$, then $d\leq a$, $d\leq b$ and $d\leq d_1$ and by exchanging $d$ and $d_1$ one gets $d_1\leq d$ and therefore $d=d_1$.
\end{proof}

\begin{ex}
Let $I$ be any set then $\Pa(I)$ is a meet semi-lattice.

\end{ex}

\subsection{Necessary condition}

\begin{prop}\label{decomposable-to-intersection}
Let $(\pi_b,b\in \A)$ be decomposable and $(s_b,b\in \A)$ its decomposition; let $v\in V$, $w\in \underset{b\in \A^{\Plus}}{\bigoplus}S_b$ such that $v=\underset{b\in \A^{\Plus}}{\sum}w_b$. For any $a,b\in\A$, $s_as_b=\delta_a(b) s_a$ and,

$$\pi_a\pi_b=\underset{\substack{c\leq a \\ c\leq b}}{\sum} s_c$$

If $\A$ is a meet semi-lattice then,

\begin{equation}\tag{I}
\forall a,b\in\A\quad  \pi_a\pi_b=\pi_{a\cap b}
\end{equation}
\end{prop}

\begin{proof}

By definition $s_as_b=\delta_a(b)s_b$,

$$\pi_a\pi_b=\underset{d\leq a}{\sum}s_d(\underset{c\leq b}{\sum} s_c)=\underset{d\leq a}{\sum}\underset{c\leq b}{\sum} \delta_d(c) s_c=\underset{\substack{c\leq a \\ c\leq b}}{\sum} s_c$$

When $\A$is a meet semi-lattice $\pi_{a\cap b}=\underset{c\leq a\cap b}{\sum} s_c=\underset{\substack{c\leq a \\ c\leq b}}{\sum} s_c$.

\end{proof}


We have defined what decomposability for a collection of endomorphisms is and shown that decomposable collections satisfy a meet or intersection property (I) when the poset is a meet semi-lattice. In the next section we will show that collections that satisfy the intersection property are decomposable and will characterize the probability measures, $\mathbb{P}$, for which $(E_{\mathbb{P}}[.|\mathcal{F}_a],a\in \mathcal{P}(I))$ admits an iteraction decomposition, i.e is decomposable.

\section{Necessary and sufficient condition for the interaction decomposition of projectors from a finite poset to $\cat{Vect}$}\label{chapitre-3-section-2}

\subsection{Zeta function and Mobius inversion with coefficient in modules}

\begin{defn}
A poset $\A$ is locally finite if for any $b\leq a$, $[b,a]=\{c\in\A :\quad b\leq c\leq a\}$ is finite.

\end{defn}

\begin{defn}[Zeta function]
Let $\A$ be a locally finite poset, let $\zeta:\underset{a\in \A}{\bigoplus}\Z\to \underset{a\in \A}{\bigoplus}\Z$ be such that for any $\lambda \in \underset{a\in \A}{\bigoplus}\Z$ and $a\in \A$,

$$\zeta(\lambda)(a)=\underset{b\leq a}{\sum} \lambda_b$$

$\zeta$ is the zeta function of $\A$.
\end{defn}

\begin{prop}[Mobius inversion]\label{Mobius-inversion}
The zeta function of a locally finite poset $\A$ is invertible, we shall note $\mu$ its inverse and there is $f:\A\times \A\to \mathbb{Z}$ such that for any $\lambda\in \bigoplus_{a\in \A}\Z$ and $a\in\A$, 

$$\mu(\lambda)(a)=\underset{b\leq a}{\sum}f(a,b)\lambda_b$$

We shall note $f$ as $\mu_\A$. 

\end{prop}
\begin{proof}
By applying Proposition 2 \cite{Rota}.

\end{proof}

\begin{defn}
Let $\A$ be a locally finite poset and $M$ a ($R$-)module; let $M_\A=\underset{a\in \A}{\bigoplus}M$. The $zeta$ function of $\A$ with values in $M$, $\zeta_{\A}(M): M_\A\to M_\A$, is such that for any $m\in M_\A$,

\begin{equation}
\zeta_{\A}(M)(m)(a)= \underset{b\leq a}{\sum}m_a
\end{equation}

We shall note $\zeta_\A(M)$ as $\zeta_\A$ making the reference to $M$ implicit.

\end{defn}

\begin{prop}\label{mobius-inversion}
Let $\A$ be a locally finite poset, $M$ a module, $\zeta_\A(M)$ is invertible, we shall call its inverse the Mobius function with values in $M$ and note it as $\mu_{\A}(M): M_{\A}\to M_{\A}$. Furthermore,for any $m\in M_\A$ and $a\in \A$,

\begin{equation}
\mu_\A(m)(a)=\underset{b\leq a}{\sum}\mu_\A(a,b)m_a
\end{equation}

\end{prop}
\begin{proof}

For any $m\in M$ and $b\in \A$, $\zeta_\A\mu_\A(m)(b)=\underset{c\leq b}{\sum}\underset{a\leq c}{\sum}\mu_\A(c,a) m_a$ and $\underset{a}{\sum}\underset{c: a\leq c\leq b}{\sum}\mu_\A(c,a) m_a= \underset{a}{\sum}\delta_b(a) m_a=m_b$ and similarly  $\mu_\A\zeta_\A=id$.

\end{proof}

\begin{defn}

Let $\A$ be a finite poset, let $(\pi_a,a\in\A)$ be a collection of endomorphism of $V$. Let $\Pi((\pi_a,a\in\A)):V\to V_{\A}$ be such that for any $v\in V$, $\Pi(v)(a)=\pi_a(v)$. For any $a\in \A$ and $v\in V$ let $s_a(v)=[\mu_\A\circ \Pi(v)](a)$, as $\zeta_{\A}\mu_\A\Pi=\Pi$,

\begin{equation}
\pi_a(v)=\underset{b\leq a}{\sum} s_b(v)
\end{equation}

\end{defn}

\begin{rem}
For $\A$ locally finite, $(s_a(v),a\in \A)$ is in general not in $V_\A$ but in $\prod_{a\in \A} V$, therefore we decide to restrict our attention to $\A$ finite for the moment.

\end{rem}

\subsection{Zeta function and functoriality }

\begin{prop}\label{commutation-zeta}
Let $\A$ be a locally finite poset, if $\B\in \U(\A)$, then one has the following commutation relations,
\begin{equation}
\zeta_{\B}\pr^{\A}_{\B}=\pr^{\A}_{\B}\zeta_{\A}
\end{equation}

\begin{equation}
\mu_{\B}\pr^{\A}_{\B}=\pr^{\A}_{\B}\mu_{\A}
\end{equation}

\end{prop}
\begin{proof}
Let $\B \in \U(\A)$, $b\in \B$ and $v\in V_\A$, 

$$\zeta_\B(pr^\A_\B(v))(b)=\underset{\substack{c\in \B\\ c\leq b }}{\sum}v_c=\underset{\substack{c\in \A \\ c\leq b }}{\sum}v_c= \zeta_\A(v)(b)$$

Therefore $\zeta_\B \pr^{\A}_{\B}=\pr^{\A}_{\B}\zeta_\A$, therefore $\pr^{\A}_{\B}=\mu_{\B}\pr^{\A}_{\B}\zeta_\A$ and $\pr^{\A}_{\B}\mu_\A=\mu_\B \pr^{\A}_{\B}$.

\end{proof}

\begin{defn}
Let $\A$ be a any poset and $\B\in \U(\A)$,

\begin{equation}
W_\A(\B)=\{v\in V_{\A} : \forall a, c \in \A,\quad \hat{a}\cap \B= \hat{c}\cap \B \implies v_a=v_c\}
\end{equation}

\end{defn}

\begin{ex}
Let $\A=\{0,1,1^{'}\}$ where $0\leq 1$ and $0\leq 1^{'}$; $W_{\hat{0}}=\{(v,v,v): \quad v\in V\}$.\\

\end{ex}

\begin{prop}\label{inclusion-eta}
Let $\A$ be a locally finite poset, for any $\B\in\U(\A)$, $\zeta_{\A}(V_\B)\subseteq W_{\A}(\B)$. If there is $b\in \A$ such that $\B=\hat{b}$ and $\A$ its a meet semi-lattice $\zeta_{\A}(V_\B)= W_{\B}$
\end{prop}

\begin{proof}

Let $v\in V_\B$, for any $a\in \A$, $\zeta_{\A}(v)(a)=\underset{c\in \hat{a}\cap \B}{\sum} v_a$. Therefore if $\hat{a}\cap \B= \hat{b}\cap \B$, $\zeta_{\A}(v)(a)=\zeta_{\A}(v)(b)$.\\

Let $u\in W_{\A}(\hat{b})$ and $v\in V_{\A}$ such that $u=\zeta_{\A}(v)$, for any $a\in \A$, $\zeta_{\A}(v1[.\in \hat{b}])(a)=\underset{c\in \hat{a}\cap \hat{b}}{\sum}v_c=u_{a\cap b}=u_a=\zeta_{\A}(v)(a)$; therefore $v=v1[.\in \hat{b}]$.\\

\end{proof}

\begin{rem}\label{counter-example-inclusion-zeta}
Let us remark that $\zeta_\A(V_\B)$ is not equal to $W_{\A}(\B)$ in general; let us consider $\A=\{a,b,c,d\}$, with, $a\leq c$, $a\leq d$, $b\leq c$, $b\leq d$. Let $\B=\{a,b\}$, asking for $v\in W_{\A}(\B)$ is the same than asking for $v_c=v_d$, which can always be fulfilled: let $u=(s_a,s_b,s_c,s_d)$ elements in $V$ such that $s_c=s_d\neq 0$, then $\zeta(u)\in W_\A(\B)$ but $u\notin V_\B$.
\end{rem}

\subsection{Intersection property implies decomposability for collection of projectors over a finite posets}

\subsubsection{For finite semi-lattices}
\begin{defn}

Let $\A$ be a non necessarily finite meet semi-lattice, let $V$ be a vector space, a collection $(\pi_a\in \Li(V),a\in \A)$ is said to verify the intersection property if,

\begin{equation}\tag{I}
\forall a,b\in \A,\quad \pi_a\pi_b=\pi_{a\cap b}
\end{equation}

\end{defn}

\begin{ex}\label{example-decomposition-intersection}
Let $\A=\{0,1,1^{'}\}$, with $0\leq 1$, $0\leq 1^{'}$; let $(V,<,>)$ be a Hilbert space, such that there are three closed subspaces $S_0,S_1,S_{1^{'}}$ such that $V=S_0\oplus_\perp S_1\oplus_\perp S_{1^{'}}$; let  $V_0=S_0$, $V_1=S_0\oplus S_1$, $V_{1^{'}}= S_0\oplus S_{1^{'}}$. Let $\pi_0,\pi_1,\pi_{1^{'}}$ be the orthogonal projection on respectively $V_0, V_1, V_{1^{'}}$, then $\pi_1\pi_{1^{'}}= \pi_0=\pi_{1\cap 1^{'}}$. 
\end{ex}

\begin{prop}\label{prop}

Let $\A$ be a finite meet semi-lattice, and let $(\pi_a\in \Li(V),a\in\A)$ that satisfies the intersection property, then for any $a,b\in \A$,

$$s_b\pi_a=1[b\leq a] s_b$$

\end{prop}

\begin{proof}
Let us remark that for any $b,c\in \A$ such that $c\leq b$, $\Pi_{\A}\circ \pi_b(c)= \Pi_{\A}(c)$; therefore for any $b\leq a$, 

$$s_b\pi_a=\pr^{\hat{a}}_{b}\pr^{\A}_{\hat{a}}\mu_{\A}\Pi_{\A}\pi_a=\pr^{\hat{a}}_{b}\mu_{\hat{a}}\pr^{\A}_{\hat{a}}\Pi_{\A}\pi_a$$

and as noted just before $\pr^{\A}_{\hat{a}}\Pi_{\A}\pi_a=\pr^{\A}_{\hat{a}}\Pi_{\A}$ so $s_b\pi_a=s_b$. Furthermore for any $b\in \A$, $\pi_b\pi_a=\pi_{a\cap b}=\pi_{a\cap b}\pi_a$, therefore $\im \Pi_\A\pi_a\subseteq W_{\hat{a}}$ and so $s_b\pi_a=1[b\leq a] s_b\pi_a$ and therefore $s_a\pi_b=1[a\leq b] s_a$.

\end{proof}


\begin{prop}\label{thm}
Let $\A$ be a finite meet semi-lattice, $V$ a vector space and $(\pi_a\in \Li(V),a\in \A)$ a collection that verifies the intersection property (I). Then for any $a,b\in \A$,

\begin{equation}
s_as_b=\delta_a(b)s_a
\end{equation}

\end{prop}

\begin{proof}

For any $a,b\in \A$,

$$s_as_b=s_a\underset{c\leq b}{\sum} \mu(b,c) \pi_c(v)=\underset{c\leq b}{\sum} \mu(b,c) s_a\pi_c(v)$$

$$s_as_b=\underset{c\leq b}{\sum}\mu(b,c) 1[a\leq c]s_a =s_a \underset{a\leq c\leq b}{\sum}\mu(b,c)$$

and $\underset{a\leq c\leq b}{\sum}\mu(b,c)=\delta_a(b)$, by definition of $\mu$.\\

\end{proof}

\begin{rem}Proposition \ref{prop} simply relies on two remarks: firstly that for any $\A$ finite and $(\pi_a\in \Li(V),a\in \A)$, if for any $b\in \A$, $\mu_\A(\im \Pi \circ \pi_b)\subseteq V_{\hat{b}}$ then for any $a,b\in \A$ such that $b\leq a$, $s_b\pi_a=1[b\leq a]s_b\pi_a$; secondly  that if $(\pi_a,a\in \A)$ is a collection of projectors, for any $b\leq a$, $s_b\pi_a=s_b$, as shown in the proof of \ref{prop}. For this reason it seems natural to define as extension of the intersection property for a collection of projectors of $V$ over any finite poset to be that $\mu_\A(\im \Pi \circ \pi_b)\subseteq V_{\hat{b}}$; this would still allow us to get Theorem \ref{thm}. Showing (I) when the poset is a join semi-lattice is however much easier than when it isn't. 
\end{rem}

\subsubsection{For finite posets}
\begin{defn}[Intersection property]\label{definition-intersection-projectors}
Let $\A$ be a finite poset, we shall say that a collection of projectors of $V$, $(\pi_a,a\in \A)$, satisfies the intersection property if,

\begin{equation}\tag{I}
\forall b\in \A,\quad \mu_\A(\im [\Pi \circ \pi_b])\subseteq V_{\hat{b}}
\end{equation}

\end{defn}

\begin{lem}\label{rem}
Condition (I) is equivalent to asking that for any $v\in V$ and any $a,b\in \A$, $\pi_a\pi_b=\underset{c\in \hat{a}\cap \hat{b}}{\sum}s_c$, and for $\A$ a finite meet semi-lattice $\pi_{a\cap b}=\sum_{c\in \hat{a}\cap \hat{b}}s_c$.
\end{lem}

\begin{prop}\label{chapitre3:thm1}
Let $\A$ be a finite poset, $(\pi_a,a\in \A)$ a collection of projectors of $V$ that verifies the intersection property; for any $a,b\in \A$,

\begin{equation}
s_as_b=\delta_a(b)s_a
\end{equation}

\end{prop}

\begin{proof}
As $\mu_\A(\im \Pi \circ \pi_b)\subseteq V_{\hat{b}}$, for any $a\in \A$ and $v\in V$, $(s_b(\pi_a(v))\in V_\B$, therefore for any $b\in \B$, $s_b\pi_a=1[b\leq a]s_b\pi_a$. For any $b,c\in \A$ such that $c\leq b$, $\Pi_{\A}\circ \pi_b(c)= \Pi_{\A}(c)$; therefore for any $b\leq a$, 

$$s_b\pi_a=\pr^{\hat{a}}_{b}\pr^{\A}_{\hat{a}}\mu_{\A}\Pi_{\A}\pi_a=\pr^{\hat{a}}_{b}\mu_{\hat{a}}\pr^{\A}_{\hat{a}}\Pi_{\A}\pi_a= \pr^{\A}_{\hat{a}}\Pi_{\A}$$

Therefore $s_b\pi_a=s_b$; the end of the proof is exactly the same than the one of Theorem \ref{thm}.

\end{proof}

\begin{rem}\label{generalisation-to-modules}
Let us call projectors of $(R-)$modules endomorphims $\pi$ of an module $M$ such that $\pi^2=\pi$; it is possible to replace the vector space $V$ by an $(R-)$modules $M$ in the definition of the intersection property (Definition \ref{definition-intersection-projectors}) and the proof of Proposition \ref{chapitre3:thm1} still holds for an $(R)-$module $M$ instead of a $(k-)$vector space $V$. We choose to start by a presentation of the equivalence theorem in the cases of vectors spaces as one usually encounters projectors as endomorphisms of vector spaces; if $\pi$ is a projector of a module $M$ then it splits, i.e. it admits a section, as its image and kernel are in direct sum, this does not happen in general for modules. 
\end{rem}

\begin{cor}\label{isomorphism-decomposable-intersection}

Let $\A$ be a finite poset, $V$ a vector space and $(\pi_a,a\in \A)$ a collection of projectors of $V$ that verifies the intersection property (I). Let for $a\in \A$, $S_a=\im s_a$, let $S(\A)= \bigoplus_{a\in \A} S_a$ and $S(\hat{a})= \bigoplus_{b\leq a} S_b$ then $\zeta_\A|^{S(\A)}_{\im \pi}$ and $\zeta_{\hat{a}}|^{S(\hat{a})}_{\im \pi_a}$ are isomorphisms.

\end{cor}

\begin{proof}

For any collection $(\pi_a\in \Li(V),a\in \A)$ one has that $\im \zeta_\A\subseteq \underset{b\in \A}{\bigoplus} S_b$ and for $a\in \A$, $\im \zeta_{\hat{a}}\subseteq \underset{b\leq a}{\bigoplus} S_b$; when $(\pi_a,a\in \A)$ is a collection of projectors that verifies (I), from Proposition \ref{chapitre3:thm1}, for any $b\leq a$,$\pi_as_b=\underset{c\leq a}{\sum}s_cs_b=s_b$ and $S_b\subseteq \im \pi_a$, so $\im \zeta_\A= \underset{b\in \A}{\bigoplus} S_b$ and $\im \zeta_{\hat{a}}=\underset{b\leq a}{\bigoplus} S_b$. As $\zeta_\A$, $\zeta_{\hat{a}}$ are injective (Proposition \ref{Mobius-inversion}), $\zeta_\A|^{\underset{b\in \A}{\bigoplus} S_b}_{\im \pi}$ and $\zeta_{\hat{a}}|^{\underset{b\leq a}{\bigoplus} S_b}_{\im \pi_a}$ are isomorphisms.

\end{proof}

\begin{thm}[Intersection property and decomposability]\label{decomposable-projector}
Let $\A$ be a finite poset, $V$ a vector space; a collection of projectors of $V$, $(\pi_a,a\in \A)$, satisfies the intersection property if and only if it is decomposable.
\end{thm}

\begin{proof}
If $(\pi_a,a\in \A)$ is decomposable Proposition \ref{decomposable-to-intersection} and Remark \ref{rem} imply that it satisfies (I).\\

For any $v\in V$, let $s_1(v)=v-\underset{a\in \A}{\sum}s_a(v)$, for any $a\in \A$, $s_as_1(v)=s_a(v)-s_a(v)=0=s_1s_a(v)$. Let $S_1=\im s_1$ then by Proposition \ref{chapitre3:thm1}, $(S_b,b\in \A^{\Plus})$ is a decomposition of $\F$ and for any $v\in V$, and $a\in \A^{\Plus} $, $\pi_a(v)=\underset{b\leq a}{\sum}s_a(v)$, where $\pi_1=id$.\\

\end{proof}

\subsection{Characterizing interaction decompositions for factor spaces}

\begin{cor}[Interaction Decomposition for factor spaces]\label{interaction-decomposition}
Let $I$ be a finite set, $(E_i,i\in I)$ a collection of finite sets, and $\mathbb{P}$ a probability measure on $E$, $(E_a[.| \mathcal{F}_a],a\in \mathcal{P}(I))$ is decomposable if and only if $\mathbb{P}$ is a product measure, i.e there is $(p_i\in \mathbb{P}(E_i),i\in I)$ such that $\mathbb{P}= \underset{i\in I}{\otimes}p_i$.\\

\end{cor}
\begin{proof}
$\mathcal{P}(I)$ possesses all its intersections; if $(\pi_a=E_a[.| \mathcal{F}_a],a\in \mathcal{P}(I))$ is decomposable, then by Corollary \ref{decomposable-projector}, for any $a,b\in \A$, $\pi_a\pi_b=\pi_{a\cap b}$. Therefore for any $i\in I$ and $f\in \mathcal{M}_b(E,\mathcal{F})$, $\pi_{\{i\}}\pi_{\{i\}^c}(f)=\pi_\emptyset(f)= E[f]$, and so $\mathbb{P}$ is a product measure.\\

Let $\mathbb{P}$ be a product measure, then for any $a,b\in \A$, $\pi_a\pi_b=\pi_{a\cap b}$ and by Theorem \ref{thm} $(E_a[.| \mathcal{F}_a],a\in \mathcal{P}(I))$ is decomposable.
\end{proof}

\section{Interaction Decomposition for presheaves in $\Mod$}\label{chapitre-3-section-3}

\subsection{The idea behind the extension and the idea of the proof}

We saw that collections of projectors can be encoded by a presheaf in Proposition \ref{inversable-1}; in this section we will extend the previous results to presheaves. Let us consider a presheaf $F:\A\to \Vect$ from a finite poset $\A$ to the category of vector spaces. Decomposable presheaves are the ones that are isomorphic to the presheaf of projections of some direct sum $\bigoplus_{a\in \A} S_a$. In order to use the previous results (Proposition \ref{chapitre3:thm1}, Theorem \ref{decomposable-projector}) we must have a way to bring back the spaces $F(b)$ into $F(a)$ for $a,b\in \A$ such that $b\leq a$ in such a way that $F^a_b$ seen inside $F(a)$ is a projector. If one considers a morphism $\pi :V\to V_1$ between two vector spaces then for any section, $s$, of $\pi$, $s\pi$ is a projector. It turns out that when considering the additional data of a functor $G:\A\to \Vect$ that is a section of $F$, in other words such that for any $b\leq a$, $F^a_bG^b_a= \id$, one can characterize intrinsically presheaves that are decomposable and there is a canonical way to define this isomorphism for such presheaves. Furthermore the presheaf of projections $F^a_b: \bigoplus_{c\leq a} S_c\to \bigoplus_{c\leq b} S_c$ has a canonical functor which morphisms are sections of the ones of $F$, so there is a functor of sections for any decomposable presheaf; this motivated the studies of such couples of functor/presheaf $(G,F)$ in what follows. In this context there is no reason to restrict our attention to vector spaces, as every $F^a_b$ is a split epimorphism by construction, i.e it admits a section, and we consider instead presheaves in $R-$ modules where $R$ is a commutative ring.

\subsection{The category of splittings}

\begin{defn}[Category $\Split$]
Let $\cat{C}$ be any category, $\Split(\cat{C})$ is the subcategory of $\cat{C}\times \cat{C}^{op}$ that has as objects $(M,M)$ with $M$ an object of $\cat{C}$ and for any $M,M_1$ two objects of $\cat{C}$, $\cat{Split}\left((M,M),(M_1,M_1)\right)$ are couples of morphisms, $(s,r)$, with $s:M\to M_1$, $r: M_1\to M$ such that,

\begin{equation}
rs= \id
\end{equation}

\end{defn}

\begin{proof}
Let $(s,r)\in \cat{Split}\left((M,M),(M_1,M_1)\right)$, $(s_1,r_1)\in \cat{Split}\left((M_1,M_1),(M_2,M_2)\right)$, $rr_1s_1s=\id$.\\

\end{proof}

\begin{nota}
We shall note $\Split(\Mod)$ simply as $\Split$.
\end{nota}

\begin{rem}
Let $\pi_1,\pi_2$ be the two projections from respectively $\cat{Split}\to \Mod$, $\cat{Split}\to \Mod^{op}$ defined as $\pi_1(V,V)=\pi_2(V,V)=V$ for $V$ and object of $\Split$ and for a morphism of $\Split$, $\pi_1(s,r)=s$, $\pi_2(s,r)=r$. Any functor $H$ from a poset $\A$ to $\Mod\times \Mod^{op}$ defines a couple of functor/presheaf $(\pi_1H,\pi_2H)$ and for any couple of functor/presheaf $(G,F)$ there is a unique a functor from $\A$ to $\Mod\times \Mod^{op}$, $H$, such that $\pi_1H=G$, $\pi_2H=F$; similarly any functor $H$ from a poset $\A$ to $\Split$ defines a couple of functor/presheaf $(\pi_1H,\pi_2H)$, any couple $(G,F)$ of functor/presheaf from $\A$ to $\Mod$ defines a functor from $\A$ to $\Split$ if and only if for any $a,b\in \A$ such that $b\leq a$, $F^a_bG^b_a=\id$. From now on when we refer to a functor from $\A\to \Split$ we shall refer to its couple $(\pi_1H,\pi_2H)$.

\end{rem}

\begin{nota}
Let $\cat{C},\cat{C}_1$ be two categories we will denote the category of functors from $\cat{C}$ to $\cat{C}_1$ as $\cat{C}_1^{\cat{C}}$.
\end{nota}

\begin{nota}
Let $\cat{C}$, $\cat{C}_1$, $\cat{C}_2$ be three categories, let $F,G:\cat{C}\to \cat{C}_1$ be two functor, $H:\cat{C}_1\to \cat{C}_2$ a functor and $\phi: F\to G$ a natural transformation. We shall note $H\star\phi$ their whiskering (Appendix A Definition A.3.5 \cite{coend}).
\end{nota}

\begin{rem}
Let $\cat{C}$ be any category, let $(G,F)$, $(G_1,F_1)$ be two functors from $\A$ to $\Split(\cat{C})$ and let $\rho=(\phi,\psi)\in \Split(\cat{C})^\A\left[(G,F), (G_1,F_1)\right]$ be a natural transformation of these functors. For any $a\in \A$, $\psi_a\phi_a= \id$ and $\phi\in \cat{C}^\A(G,G_1)$, $\psi\in \cat{C}^\A(F_1,F)$, in other words $\pi_1\star\rho$ is a natural transformation from $G$ to $G_1$ and $\pi_2\star\rho$ is a natural transformation from $F_1$ to $F$. Now if $\phi$ or $\psi$ was an isomorphism then $(\phi,\phi^{-1})$ would be a morphisms from $(G,F)$ to $(G_1,F_1)$ and respectively $(\psi^{-1}, \psi)$; furthermore if $\phi$ is an isomorphism it also a natural transformation from $F$ to $F_1$ and respectively if $\psi$ is an isomorphism it is a natural transformation from $G$ to $G_1$; when one of the two options is true the reference to its inverse will be implicit, for example we will just say that $\phi$ is a natural transformation from $(G,F)$ to $(G_1,F_1)$. 
\end{rem}

\begin{prop}
Let $(G,F),(G_1,F_1)$ be two functors from $\A$ to $\Split$ and $\phi:(G,F)\to (G_1,F_1)$ be a natural transformation. Let $\im \phi=(\im \pi_1\star \phi, \im \pi_2\star \phi)$, $\im\phi$ is a functor from $\A$ to $\Split$.

\end{prop}

\begin{proof}
Let $g=\im \pi_1\star\phi$, $f=\im\pi_2\star\phi$, for any $a,b\in \A$ such that $a\geq b$, and $v\in g(a)$, $f^a_bg^b_a(v)= {F_1}^a_b{G_1}^b_a(v)= v$.\\

\end{proof}

\begin{prop}
Let $I$ be any set and $\left((G_i,F_i), i\in I\right)$ a collection of functors from $\A$ to $\Split$; $(\underset{i\in I}{\bigoplus}G_i,\underset{i\in I}{\bigoplus}F_i)$ and $(\underset{i\in I}{\prod}G_i,\underset{i\in I}{\prod}F_i)$ are functors from $\A$ to $\Split$.

\end{prop}
\begin{proof}
Let $g=\underset{i\in I}{\prod}G_i$, $f=\underset{i\in I}{\prod}F_i$, let $a,b\in \A$ such that $b\leq a$, let $v\in  \underset{i\in I}{\prod}G_i(a)$, for any $j\in I$, $f^a_bg^b_a(v)(j)= {F_j}^a_b{G_j}^b_a(v_j)=v_j$.
\end{proof}

\begin{prop}\label{rel-split}
Let $(G,F)$ be a functor from a poset $\A$ to $\Split$, for any $a\geq b\geq c$,

$$F^a_bG^c_a=G^c_b$$

\end{prop}

\begin{proof}
$F^a_b G^c_a= F^a_b G^b_a G^c_b= G^c_b$.

\end{proof}

\begin{prop}\label{epi-to-split}
Let $(G,F)$ be a functor from $\A$ to $\Split$ and $(G_1,F_1)$ a functor from $\A$ to $\Mod\times \Mod^{op}$; if there is an epimorphism $\phi$ from $(G,F)$ to $(G_1,F_1)$ then $(G_1,F_1)$ is a functor from $\A$ to $\Split$.
\end{prop}

\begin{proof}
Let $a,b\in \A$ such that $b\leq a$, one has that ${F_{1}}^a_b {G_1}^b_a\phi_b= \phi_b F^a_b G^b_a= \phi_b$ and as $\phi_b$ is an epimorphism, ${F_{1}}^a_b {G_1}^b_a=\id$.
\end{proof}

\subsection{Decomposable functors to $\Split$}

\begin{defn}
Let $\A$ be any poset and $G:\A\to \Mod$ be a functor. For any $a,b,c\in \A$, let $G1[a\leq .](b)=G(b)$ if $b\geq a$ and otherwise $G1[a\leq .](b)=0$; for $a\leq c\leq b$, $G1[a\leq .]^c_b= G^c_b$ if $a\not\leq c$ and $c\leq b$, $G1[a\leq . ]^c_b=0$ (as $0$ is initial in $\Mod$).  Similarly let $F$ be a presheaf from $\A$ to $\Mod$, let for any $a,b\in \A$ and $a\leq b$, $F1[a\leq .](b)=F(b)$ and otherwise $F1[a\leq .](b)=0$, for $a\leq c\leq b$, $F1[a\leq .]^b_c= F^b_c$.

\end{defn}

\begin{prop}
Let $G:\A\to \Mod$ be a functor and $F:\A^{op}\to \Mod$ be a presheaf, then $G1[a\leq .]$ is a functor and $F1[a\leq .]$ is a presheaf. 
\end{prop}
\begin{proof}

Let $a,b,c,d\in \A$ such that $d\leq c\leq b$, if $a\leq d$ then $G1[a\leq .]^c_bG1[a \leq .]^d_c= G^c_bG^d_c=G^d_b= G1[a\leq .]^d_b$ and $F1[a\leq .]^c_dF1[a\leq .]^b_c= F^c_dF^b_c=F1[a\leq .]^b_d$; if $a\not\leq d$ then $G1[a\leq .]^c_bG1[a \leq .]^d_c=G1[a\leq .]^c_b0= G1[a\leq .]^d_b$, $F1[a\leq .]^c_dF1[a\leq .]^b_c= 0F1[a\leq  .]^b_c=F1[a\leq .]^b_d$.\\

\end{proof}

\begin{rem}
Let $G$ be functor from a poset $\A$ to a $\core(\Mod)$, let $(G^a:G(a)\to \colim G, a\in \A)$ be the colimit cone over $G$; $(G^a, a\in \A)$ is a monomorphism from $G$ to the constant functor $\colim G$. Therefore $G$ is isomorphic to the image of $(G^a, a\in \A)$, that we denote as $G_1$. Let us call the connected components of $\A$ the equivalence classes for the equivalence relation generated by the order on $\A$; for any connected component of $\A$, $C$, there is a module $M_C$ such that $G_1|_{C}$ is the constant functor $M_C$.
\end{rem}

\begin{defn*}[Decomposable functors to $\Split$]\label{Decomposable-presheaf}
Let $H$ be a functor from $\A$ to $\Split$. $H$ is decomposable when there is a collection $((G_a,F^a),a\in \A)$ of functors from $\A$ to $\Split(\core \Mod)$ such that,

\begin{equation}
H\cong (\prod_a G_a 1[a\leq .], \prod_a F_a 1[a\leq .])
\end{equation}

When $H$ is decomposable we shall call $(\prod_{a\in \A}G_a1[a\leq .],\prod_{a\in \A}F^a1[a\leq .])$ its decomposition and note it as $(\prod_{a\in \A}S_a,\prod_{a\in \A}S^a)$.
\end{defn*}

\begin{rem}
A presheaf from $\A$ to $\Mod$ is decomposable when there is a collection of presheaves, $W^a$, from $\A$ to $\Split(\core \Mod)$ such that $F$ is isomorphic to $\prod_{a\in \A} W^a1[a\leq .]$; in other words a presheaf is decomposable when there is a section $G$ of $F$ such that $(G,F):\A\to \Split$ is decomposable. In the example of the interaction decomposition for factor spaces, $G$ is the poset of factor subspaces indexed by the parts of a finite set and the presheaves that are considered are the projectors on the subspaces for a given scalar product; decomposability of $G$ is the subject of \cite{GS2}, decomposable couples $(G,F)$ are seen as specific projectors on $G$. This is an other motivation to the fact that we focus more on the definition of decomposable functors from a poset to $\Split$ than on the definition of decomposable presheaves.
\end{rem}

\begin{cor}\label{decomposable-components}
Let $H$ be a decomposable functor from $\A$ to $\Split$ and $(\prod_{a\in \A}S_a,\prod_{a\in \A}S^a)$ its decomposition, $(\prod_{a\in \A}S_a,\prod_{a\in \A}S^a)$ is a functor from $\A$ to $\Split$ and for any $a\in \A$, $(S_a,S^a)$ is too.

\end{cor}
\begin{proof}
By Proposition \ref{epi-to-split}, $(\prod_{a\in \A}S_a,\prod_{a\in \A}S^a)$ is a functor from $\A$ to $\Split$. Let $a,b,c\in \A$ such that $c\leq b$ and $v\in \underset{d\in \A}{\prod}S^d(c)$, $\underset{d\in \A}{\prod}{S^d}^b_c \underset{d\in \A}{\prod}{S_d}^c_b(v)(a)=v_a= {S^a}^b_c{S_a}^c_b(v_a)$. 
\end{proof}

\begin{rem}
Let $H$ be a decomposable functor from $\A$ to $\Split$ and $(\prod_{a\in \A}S_a,\prod_{a\in \A}S^a)$ its decomposition, for any $a,b,c\in \A$ such that $c\leq b$, ${{S_a}^c_b}^{-1}={{S^a}^b_c}$.

\end{rem}

\subsection{Bringing back $F|_{\hat{a}}$ in $F(a)$}

\begin{prop}\label{def-right-coupling}
Let $\A$ be a poset and $(G,F)$ a functor from $\A$ to $\Split$, for any $a,b,c\in \A$ such that $c\leq b\leq a$,

\begin{equation}\label{good-couple-1}
F^b_c (ker G^b_a) \subseteq \ker G^c_a\quad \quad F^a_b(\im G^c_a)= \im G^c_b
\end{equation}
\end{prop}

\begin{proof}
For any $a,b,c\in\A$ such that $c\leq b\leq a$, $F^a_bG^b_a=\id$, therefore $\ker G^b_a=0$; $F^a_bG^c_a=G^c_b$ (Proposition \ref{rel-split}) therefore $F^a_b(\im G^c_a)= \im G^c_b$.\\

\end{proof}

\begin{defn}
Let $\A$ be a poset and $(G,F)$ a functor from $\A$ to $\Split$, let $R(\alpha,a)=\im G^a_{\alpha}$, let $\alpha\geq b\geq c$, let us call $R^{ \alpha b }_{ \alpha c}: \im G^b_{\alpha} \rightarrow \im G^c_{\alpha}$ the unique morphism that satisfies $R^{\alpha b}_{\alpha c}{G^b_{\alpha}}|^{R(\alpha, b)} ={G^c_{\alpha}}|^{R(\alpha, c)} F^b_c$ and for $\alpha \geq \beta \geq a$, let $R^{\alpha a}_{\beta a}: R(\alpha, a) \to R(\beta, a)$ be such that $R^{\alpha a}_{\beta a}={F^{\alpha}_{\beta}}|^{R(\beta,a)}_{R(\alpha, a)}$.\\

For $\alpha\geq \beta \geq a\geq b$, $R(\alpha,b)\subseteq R(\alpha,a)$, we shall note the inclusion as $L^{\alpha b}_{\alpha a}$; let $L^{\beta a}_{\alpha a}=G^{\beta}_{\alpha}|^{R(\alpha,a)}_{R(\beta,a)}$.\\

We shall call $L(G,F)$ the left coupling of $(G,F)$ and $R(G,F)$ its right coupling.

\end{defn}

\begin{proof}
By Proposition \ref{def-right-coupling} $R$ is well defined. For any $a\geq b\geq c$, $G^{b}_{a}G^{c}_{b}=G^{c}_{a}$, therefore $R(a,c)=\im G^{c}_{a}\subseteq \im G^{b}_{a}= R(a,b)$, and $G^b_a(\im G^c_b)=  \im G^c_a$ and $L$ is well defined.
\end{proof}


\begin{rem}
Let $(G,F)$ be a functor from $\A$ to $\Split$, for any $\alpha \in \A$, $G_{|\hat{\alpha}}$ induces a functor monomorphism $G^{\hat{\alpha}}_\alpha: G_{|\hat{\alpha}}\to G(a)$ and a presheaf monomorphism ${G^{\hat{\alpha}}_\alpha}^{F}:F_{\hat{\alpha}}\to  F(\alpha)$; $L(\alpha,.)=\im G^{\hat{\alpha}}_\alpha$, $R(\alpha,.)=\im{ {G^{\hat{\alpha}}_\alpha}^{F} }$. Furthermore, $G^{\hat{\alpha}}_\alpha|^{L(\alpha,.)}$ and ${G^{\hat{\alpha}}_\alpha}^{F} |^{R(\alpha,.) }$ are isomorphisms. 
\end{rem}

\begin{prop} \label{functor-extension}
Let $(G,F)$ be a functor from $\A$ to $\Split$, $L$ its left coupling, $R$ its right coupling; for any  $\alpha,\beta,a,b\in \A$ such that $\alpha\geq \beta \geq a\geq b$,

\begin{equation}
L^{\beta a}_{\alpha a}L^{\beta b}_{ \beta a}= L^{ \alpha b}_{ \alpha a}L^{\beta b}_{\alpha b}
\end{equation}

\begin{equation}\label{good-couple-2}
R^{\beta a}_{\beta b}R^{\alpha a}_{\beta a}=R^{\alpha b}_{\beta b}R^{\alpha a}_{\alpha b}
\end{equation}

\end{prop}
\begin{proof}

Let $\alpha\geq \beta\geq a\geq b$, for any $v\in R(\beta,b)$, $L^{\beta a}_{\alpha a}L^{\beta b}_{\beta a}(v)=G^{\beta}_{\alpha}(v)= L^{\alpha b}_{ \alpha a}L^{\beta b}_{\alpha b}(v)$.\\

For $v\in F(a)$, $v_1\in F(b)$, $R^{\alpha a}_{\beta a}G^a_\alpha(v)= F^\alpha_\beta G^a_\alpha(v)= G^a_\beta(v)$ and $ R^{\alpha b}_{\beta b} G^b_\alpha(v_1)=G^b_\beta (v_1)$; $R^{\beta a}_{\beta b}R^{\alpha a}_{\beta a}G^a_\alpha(v)=R^{\beta a}_{\beta b}G^a_\beta(v)$, $R^{\alpha b}_{\beta b} R^{\alpha a }_{\alpha b} G^a_\alpha (v)= R^{\alpha b}_{\beta b} G^b_\alpha F^a_b(v)=G^b_\beta F^a_b(v)$, by construction $G^b_\beta F^a_b(v)= R^{\beta a}_{\beta b}G^a_\beta $.

\begin{equation}
\begin{tikzpicture}[baseline=(current  bounding  box.center),node distance=2cm, auto]
\node (A) {$F(a)$ };
\node (B) [right of=A] {$R(\alpha,a)$};
\node (C) [right of =B] {$R(\beta,a)$};
\node (D) [below of=A ] {$F(b)$};
\node (E) [right of =D] {$R(\alpha,b)$};
\node (F) [right of =E] {$R(\beta,b)$};
\draw[->] (A) to node {$G^a_{\alpha}|^{R(\alpha,a)}$} (B);
\draw[->] (B) to node {$R^{\alpha a}_{\beta a}$} (C);
\draw[->] (D) to node [below] {$G^b_{\alpha}|^{R(\alpha,b)}$} (E);
\draw[->] (E) to node [below] {$R^{\alpha b}_{\beta b}$} (F);
\draw[->] (A) to node {$F^a_b$} (D);
\draw[->] (B) to node {$R^{\alpha a}_{\alpha b}$} (E);
\draw[->] (C) to node {$R^{ \beta a}_{\beta b}$} (F);
\end{tikzpicture}
\end{equation}

\end{proof}

\begin{defn}

Let $\A_1$ be the subposet of $\A\times \A$ constituted of couples $(\alpha, a)$ such that $a\leq \alpha$.

\end{defn}

\begin{prop}\label{extension-functor-poset}
Let $\A$ be a poset, $\cat{C}$ be any category; let $M_1=\{\left((\alpha,a),(\alpha,b)\right): (\alpha,a),(\alpha,b)\in \A_1 \text{ and } a\geq b\}$, $M_2=\{\left((\alpha,a),(\beta,a)\right): (\alpha,a),(\beta,a)\in \A_1 \text{ and } \alpha\geq \beta\}$, and let $(G^{i}_{j}; i,j\in M_1\cup M_2: i\leq j)$ be such that for any $(\alpha,a),(\alpha,b),(\alpha,c)\in \A_1$ such that $(\alpha,a)\geq (\alpha,b)\geq (\alpha,c)$, 
 $$G^{\alpha b}_{\alpha a}G^{\alpha c}_{\alpha b}= G^{\alpha c}_{\alpha a}$$
 
 for any $(\alpha,a),(\beta,a),(\gamma,a)$ such that $(\alpha,a)\geq (\beta,a)\geq (\gamma,a)$,
 
 $$G^{\beta a}_{\alpha a} G^{\gamma a}_{\beta a}= G^{\gamma a}_{\alpha a}$$

and for any $(\alpha, a),(\alpha, b),(\beta, a),(\beta, b)$ such that $(\alpha, a)\geq (\alpha, b)\geq (\beta, b)$ and $(\alpha, a)\geq (\beta, a)\geq (\beta, b)$, i.e $\alpha\geq \beta \geq a\geq b$,

$$G^{\beta a}_{\alpha a} G^{\beta b}_{\beta a}= G^{\alpha b}_{\alpha a}G^{\beta b}_{\alpha b}$$

Then $G$ extends into a unique functor $G_1:\A_1\to \cat{C}$, we shall also denote this extension as $G$.

\end{prop}

\begin{proof}
Let us remark that for any $(\alpha, a),(\beta,b)\in \A_1$ such that $(\alpha,a)\geq (\beta,b)$, then $(\alpha,a)\geq (\alpha,b)\geq (\beta,b)$. Let $G_1:\A_1\to \cat{C}$ be a functor such that for any $(\alpha,a)\geq (\alpha, b)$, ${G_1}^{\alpha a}_{\alpha b}=G^{\alpha a}_{\alpha b}$ and for any $(\alpha,a)\geq (\beta,a)$, ${G_1}^{\beta a}_{\alpha a}=G^{\beta a}_{\alpha a}$, then for any $(\alpha,a)\geq (\beta,b)$, ${G_1}_{\alpha a}^{\beta b}= G^{\alpha b}_{\alpha a}G^{\beta b}_{\alpha b}$. Therefore there can be only one functor that extends $G$ to $\A_1$.\\

Let for any $(\alpha,a)\geq (\beta,b)$, ${G_1}_{\alpha a}^{\beta b}= G^{\alpha b}_{\alpha a} G^{\beta b}_{\alpha b}$.\\

\begin{equation}
\begin{tikzpicture}[baseline=(current  bounding  box.center),node distance=2cm, auto]
\node (A) {$G(\beta,b)$ };
\node (B) [below of=A] {$G(\alpha, b)$};
\node (C) [right of =B] {$G(\alpha, a)$};
\draw[->] (A) to node [left]{$G^{\beta b}_{\alpha b}$} (B);
\draw[->] (B) to node [below]{$G^{\alpha b}_{\alpha a}$} (C);
\draw[->] (A) to node {${G_1}^{\beta b}_{\alpha a}$} (C);
\end{tikzpicture}
\end{equation}

For any $(\alpha,a)\geq (\beta,b)\geq (\gamma, c)$,

$${G_1}^{\beta b}_{\alpha a}{G_1}^{\gamma c}_{\beta b}= G^{\alpha b}_{\alpha a} G^{\beta b}_{\alpha b}G^{\beta c}_{\beta b}G^{\gamma c}_{\beta c}= G^{\alpha b}_{\alpha a} G^{\alpha c}_{\alpha b}G^{\beta c}_{\alpha c}G^{\gamma c}_{\beta c}= G^{\alpha c}_{\alpha a}G^{\gamma c}_{\alpha c}$$

\begin{equation}
\begin{tikzpicture}[baseline=(current  bounding  box.center),node distance=2cm, auto]
\node (A) {$G(\gamma,c)$ };
\node (B) [below of=A] {$G(\beta,c)$};
\node (C) [below of =B] {$G(\alpha,c)$};
\node (D) [right of=B ] {$G(\beta,b)$};
\node (E) [right of =C] {$G(\alpha,b)$};
\node (F) [right of =E] {$G(\alpha,a)$};
\draw[->] (A) to node [left] {$G^{\gamma c}_{\beta c} $} (B);
\draw[->] (B) to node [left] {$G^{\beta c}_{\alpha c}$} (C);
\draw[->] (B) to node {$G^{\beta c}_{\beta b}$} (D);
\draw[->] (D) to node [left]{$G^{\beta b}_{\alpha b}$} (E);
\draw[->] (C) to node [below]{$G^{\alpha c}_{\alpha b}$} (E);
\draw[->] (E) to node [below]{$G^{\alpha b}_{\alpha a}$} (F);
\draw[->] (A) to node {${G_1}^{\gamma c}_{\beta b}$} (D);
\draw[->] (D) to node {${G_1}^{\beta b}_{\alpha a}$} (F);
\end{tikzpicture}
\end{equation}

\end{proof}


\begin{rem}\label{Presheaf-extension}
Let $\A$ be a poset, if one applies Proposition $\ref{extension-functor-poset}$   for $\A^{op}$ then one extends $G$ to a presheaf. 

\end{rem}

\begin{cor}\label{extension-left-right}
Let $(G,F)$ be a functor from $\A$ to $\Split$, let $(L,R)$ be its left and right coupling; $(L,R)$ has a unique extension into a functor from $\A$ to $\Split$.

\end{cor}

\begin{proof}
By construction for any $(\alpha,a), (\beta,a),(\gamma,a)\in \A_1$ such that $(\alpha,a)\geq (\beta,a)\geq (\gamma,a)$, $L^{\beta a}_{\alpha a}L^{\gamma a}_{\beta a}= L^{\gamma a}_{\alpha a}$, and for any $(\alpha, a),(\alpha, b), (\alpha, c)\in \A_1$ such that $(\alpha, a)\geq (\alpha, b)\geq (\alpha, c)$, $L^{\alpha b}_{\alpha a}L^{\alpha c}_{\alpha b}= L^{\alpha c}_{\alpha a}$. Therefore by Proposition \ref{functor-extension} and Proposition \ref{extension-functor-poset}, $L$ extends into a unique functor from $\A_1$ to $\Mod$. \\

By construction for any $(\alpha,a), (\beta,a),(\gamma,a)\in \A_1$ such that $(\alpha,a)\geq (\beta,a)\geq (\gamma,a)$, $R^{\beta a}_{\gamma a}R^{\alpha a}_{\beta a}=R^{\alpha a}_{\gamma a}$. Let $(\alpha, a),(\alpha, b), (\alpha, c)\in \A_1$ such that $(\alpha, a)\geq (\alpha b)\geq (\alpha c)$ and $v\in G(a)$, $R^{\alpha b}_{\alpha c} R^{\alpha a}_{\alpha b} G^a_\alpha(v) = R^{\alpha b}_{\alpha c} G^{b}_\alpha F^a_b(v)= G^c_\alpha F^a_c(v)= R^{\alpha a}_{\alpha c} G^a_\alpha(v)$; as $G^a_\alpha$ is surjective, $ R^{\alpha b}_{\alpha c} R^{\alpha a}_{\alpha b} = R^{\alpha a}_{\alpha c}$. Therefore by Proposition \ref{functor-extension} and Proposition \ref{extension-functor-poset}, $R$ extends into a unique presheaf from $\A_1$ to $\Mod$.\\

For $\alpha\geq a\geq b$ and $v\in G(b)$, $R^{\alpha a}_{\alpha b}L^{\alpha b}_{\alpha a}G^b_\alpha(v)=R^{\alpha a}_{\alpha b}G^a_\alpha G^b_a(v)= G^b_\alpha F^a_b G^b_a(v)= G^b_\alpha(v)$, as $G^b_\alpha$ is surjective, $R^{\alpha a}_{\alpha b}L^{\alpha b}_{\alpha a}=\id $; for $\alpha \geq \beta \geq a$ and $v\in R(\beta,a)$, $R^{\alpha a}_{\beta a}L^{\beta a}_{\alpha a}(v)=F^\alpha_\beta G^\beta_\alpha(v)=v$; therefore for any $(\alpha,a)\geq (\beta,b)$,
\begin{equation}
R^{\alpha a}_{\beta b} L^{\beta b}_{\alpha a}= R^{\alpha b}_{\beta b}R^{\alpha a}_{\alpha b} L^{\alpha b}_{\alpha a}L^{\beta b}_{\alpha b}= R^{\alpha b}_{\beta b}L^{\beta b}_{\alpha b}=\id
\end{equation}

\end{proof}

\begin{rem}\label{morphism-functor-to-extension}
For $(G,F)$ a functor from $\A$ to $\Split$, and $(\alpha, a)\in \A_1$, $L^{a a}_{\alpha \alpha}= G^a_{\alpha}$, $R^{\alpha \alpha}_{a a}=F^{\alpha}_{a}$. Indeed, let us recall that $R(a,a)=F(a)$, $R^{\alpha \alpha}_{\alpha a}= L^{a a}_{\alpha a} F^{\alpha}_a$ and $R^{\alpha \alpha}_{a a}=R^{\alpha a}_{a a}R^{\alpha \alpha}_{\alpha a}= R^{\alpha a}_{a a}L^{a a}_{\alpha a} F^{\alpha}_a$; $R^{\alpha a}_{a a}=F^{\alpha}_{a} |_{R(\alpha, a)}=F^{\alpha}_{a}L^{\alpha a}_{\alpha \alpha}$; so $R^{\alpha \alpha}_{a a}= F^{\alpha}_{a}L^{\alpha a}_{\alpha \alpha}L^{a a}_{\alpha a} F^{\alpha}_a= F^{\alpha}_a L^{a a}_{\alpha\alpha}F^{\alpha}_a=F^{\alpha}_a G^a_\alpha F^\alpha_a=F^\alpha_a$.

\end{rem}

\begin{cor}
Let $(G,F)$ be a functor from $\A$ to $\Split$, for any $\alpha,\beta,a\in \A$ such that $\alpha \geq \beta \geq a$, $L^{\beta a}_{\alpha a}$ is an isomorphism, its inverse is $R^{\alpha a}_{\beta a}$.

\end{cor}

\begin{proof}
For any $\alpha \geq \beta \geq a$, $L^{\beta a}_{\alpha a}= G^\beta_\alpha|^{\im G^a_\alpha}_{\im G^a_\beta}$ and as $G^\beta_\alpha$ is injective and $G^\beta_\alpha G^a_\beta= G^a_\alpha$, $\im L^{\beta a}_{\alpha a}= \im G^a_\alpha= L(\alpha,a)$; therefore $L^{\beta a}_{\alpha a}$ is an isomorphism. Furthermore by Corollary \ref{extension-left-right} $R^{\alpha a}_{\alpha b}$ is the inverse of $L^{\alpha b}_{\alpha a}$.

\end{proof}

\begin{defn}
Let $\A$ be any poset, let $(G,F)$ be a functor from $\A$ to $\Split$. For any $(\alpha,a)\in \A_1$, let $V(\alpha, a)=\underset{b\leq a}{\prod}G(\alpha)$ (which in the previous section we would note as $G(\alpha)_{\hat{a}}$). For any $\alpha,\beta,a,b$ such that $\alpha\geq \beta\geq a\geq b$ let ${V_r}^{\alpha a}_{\alpha b} = pr^{\underset{c\in \hat{a}}{\prod} G(\alpha)}_{\underset{c\in \hat{b}}{\prod} G(\alpha)}$, ${V_l}^{\alpha b}_{\alpha a}=  i^{\underset{c\in \hat{b}}{\prod} G(\alpha)}_{\underset{c\in \hat{a}}{\prod} G(\alpha)}$, ${V_r}^{\alpha a}_{\beta a}:V(\alpha,a)\to V(\beta,a)$ be such that for any $v\in V(\alpha,a)$ and $c\leq a$, ${V_r}^{\alpha a}_{\beta a}(v)(c)=F^{\alpha}_{\beta}(v_c)$, ${V_l}^{\beta a}_{\alpha a}:V(\beta,a)\to V(\alpha,a)$ be such that for any $v\in V(\beta,a)$ and $c\leq a$, ${V_l}^{\alpha a}_{\beta a}(v)(c)=G^{\beta}_{\alpha}(v_c)$. 

\end{defn}

\begin{prop}
Let $\A$, $(G,F)$ be a functor from $\A$ to $\Split$, $(\prod_{a\in \A}G1[a\leq .], \prod_{a\in \A}F1[a\leq.])$ is a functor from $\A$ to $\Split$. $(V_l, V_r)$ extends into a unique functor form $\A_1$ to $\Split$.

\end{prop}

\begin{proof}

Let $b,c\in \A$ such that $c\leq b$, let us note $\prod_{a\in \A}F1[a\leq.]^b_c$ as $F$ and $\prod_{a\in \A}G1[a\leq .]^c_b$ as $G$; for any $v\in \prod_{a\in \A}G1[a\leq .](c)$, and $a\in \A$, $FG(v)(a)= F^b_cG^c_b(v_a) 1[a\leq b] 1[a\leq c]= v_a 1[a\leq c]= \id_{\prod_{a\in \A}G1[a\leq .](c)}(v)(a)(\cong \id_{\underset{\substack{a_1\in \A \\ a_1\leq c}}{\prod}G(c)}(v)(a))$ (this also shows that $(G1[a\leq.],F1[a\leq])$ is a functor from $\A$ to $\Split$). \\

For $\alpha \geq a\geq b\geq c$, ${V_l}^{\alpha b}_{\alpha a}{V_l}^{\alpha c}_{\alpha b}= {V_l}^{\alpha c}_{\alpha a}$, ${V_r}^{\alpha b}_{\alpha c}{V_r}^{\alpha a}_{\alpha b} = {V_r}^{\alpha a}_{\alpha c}$. For any $\alpha \geq \beta \geq \gamma\geq a$, any $v\in V(\alpha ,c)$, any $c\leq a$, ${V_l}^{\beta a}_{\alpha a}{V_l}^{\gamma a}_{\beta a}(v)(c)=G^\beta_\alpha G^\gamma_\alpha(v_c)=G^\gamma_\alpha(v_c)= {V_l}^{\gamma a}_{\alpha a}(v)(c)$ and for any $v\in V(\alpha, a)$, and $c\leq a$, ${V_r}^{\beta a}_{\gamma a}{V_r}^{\alpha a}_{\beta a}(v)(c)= F^\beta_\gamma F^\alpha_\beta(v_c)= F^\alpha_\gamma(v_c)={V_r}^{\alpha a}_{\gamma a}(v)(c)$.\\

Let $\alpha \geq \beta \geq a\geq b$, $v\in V(\beta, b)$; for any $c\leq a$, ${V_l}^{\beta a}_{\alpha a}{V_l}^{\beta b}_{\beta a}(v)(c)= G^\beta_\alpha({V_l}^{\beta b}_{\beta a}(v)(c))= G^\beta_\alpha(v_c 1[c\leq b])=G^\beta_\alpha(v_c) 1[c\leq b]$ and ${V_l}^{\alpha b}_{\alpha a}{V_l}^{\beta b}_{\alpha b}(v)(c)= {V_l}^{\beta b}_{\alpha b}(v)(c) 1[c\leq b]= G^\beta_\alpha(v_c) 1[c\leq b]$;
 similarly one has that ${V_r}^{\alpha b}_{\beta b}{V_r}^{\alpha a}_{\alpha b}= {V_r}^{\beta a}_{\beta b} {V_r}^{\alpha a}_{\beta a }$. Therefore by Proposition \ref{extension-functor-poset} $V_l$ extends to a functor from $\A_1$ to $\Mod$ and $V_r$ to a presheaf from $\A_1$ to $ \Mod$.\\ 

For any $\alpha \geq \beta\geq a$, $v\in V(\beta,a)$, $c\leq b$, ${V_r}^{\alpha a}_{\beta a}{V_l}^{\beta a}_{\alpha a}(v)(c)= F^\alpha_\beta G^\beta_\alpha(v_c)=v_c$; for any $\alpha \geq a\geq b$, ${V_r}^{\alpha a}_{\alpha b}{V_l}^{\alpha b}_{\alpha a}=\id$, therefore for any $(\alpha,a) \geq (\beta,b)$, ${V_r}^{\alpha a}_{\beta b}{V_l}^{\beta b}_{\alpha a}= {V_r}^{\alpha b}_{\beta b}{V_r}^{\alpha a}_{\alpha b} {V_l}^{\alpha b}_{\alpha a}{V_l}^{\beta b}_{\alpha b}= \id$.

\end{proof}

Until now in this subsection there was no constraint on $\A$, in order to be able to define $\zeta_{\hat{\alpha}}(G(\alpha))$ on $V(\alpha,\alpha)$ we will have to assume that $\hat{\alpha}$ is finite for any $\alpha\in \A$.\\

\begin{nota}\label{pf}
The class of posets that are such that $\hat{a}$ is finite for any $a\in \A$ will be denoted as $\Pf$.
\end{nota}

\begin{rem}
Let $\A\in \Pf$ and $(G_a,a\in \A)$ a collection of functors from $\A$ to $\Mod$, $\prod_{a\in \A}G_a1[a\leq .]= \underset{a\in \A}{\bigoplus}G_a1[a\leq.]$, indeed for any $b\in \A$, $\underset{a: a\leq b}{\prod}G_a(b)=\underset{a: a\leq b}{\bigoplus}G_a(b)$.
\end{rem}

\begin{prop}\label{zeta-extented}
Let $(G,F)$ be a functor from $\A\in \Pf$ to $\Split$, for any $(\alpha, a)\in \A_1$, let $\zeta(\alpha,a)=\zeta_{\hat{a}}(G(\alpha)):V(\alpha,a)\to V(\alpha,a)$ and $\mu(\alpha, a)= \mu_{\hat{a}}(G(\alpha)):V(\alpha,a)\to V(\alpha,a)$; $\zeta$, $\mu\in \Mod^{\A^{op}}\left(V_r,V_r\right)$.
\end{prop}

\begin{lem}\label{lemme-eta}
Let $V$,$V_1$ be two modules, $\A$ a finite poset, and $l:V\to V_1$ a linear application; let $L:V_\A\to {V_1}_\A$ be such that $L(v)(a)=l(v_a)$. $\zeta_{\A}(V_1)L=L\zeta_{\A}(V)$ and $\mu_{\A}(V_1)L=L\mu_{\A}(V)$.

\end{lem}

\begin{proof}
For any $v\in V_\A$, $a\in \A$, $L(\zeta_{\A}(V)(v))(a)=l(\underset{b\leq a}{\sum}v_b)=\underset{b\leq a}{\sum}l(v_b)=\zeta_{\A}(V_1)(L(v))(a)$. Furthermore $\mu_{\A}(V_1)\zeta_{\A}(V_1)L\mu_{\A}(V) =\mu_{\A}(V_1)L\zeta_{\A}(V)\mu_{\A}(V)$ so $\mu_{\A}(V_1)L=L\mu_{\A}(V)$.
\end{proof}

\begin{proof} Proof of Proposition \ref{zeta-extented}.\\

For any $\alpha\geq a \geq b$, by Proposition \ref{commutation-zeta} $\zeta(\alpha,b){V_r}^{\alpha a}_{\alpha b}= {V_r}^{\alpha a}_{\alpha b} \zeta(\alpha,a)$. By Lemma \ref{commutation-zeta}, for any $\alpha \geq \beta\geq a$, $\zeta(\beta,a){V_r}^{\alpha a}_{\beta a}= {V_r}^{\alpha a}_{\beta a}\zeta(\alpha,a)$. Therefore for any $(\alpha,a),(\beta,b) \in \A_1$ such that $(\alpha,a)\geq (\beta,b)$, ${V_r}^{\alpha a}_{\beta b}\zeta(\alpha,a)= {V_r}^{\alpha b}_{\beta b}V^{\alpha a}_{\alpha b}\zeta(\alpha,a)= {V_r}^{\alpha b}_{\beta b}\zeta(\alpha,b) {V_r}^{\alpha a}_{\alpha b}=\zeta(\beta, b){V_r}^{\alpha b}_{\beta b}{V_r}^{\alpha a}_{\alpha b}$. Therefore ${V_r}^{\alpha a}_{\beta b}\mu(\alpha,a)= \mu(\beta,b){V_r}^{\alpha a}_{\beta b}$.

\end{proof}

\begin{rem}
$\zeta$, $\mu$ are in general not natural transformations from $V_l$ to $V_l$ because for $\alpha \geq a\geq b$, ${V_l}^{\alpha b}_{\alpha a}\zeta(\alpha, b)\neq \zeta(\alpha, a){V_l}^{\alpha b}_{\alpha a}$.
\end{rem}

\begin{nota}
For any $\alpha\in \A$, we shall note $\zeta(\alpha,\alpha)$ as $\zeta^{\alpha}$.
\end{nota}

\subsection{Intersection for functors from $\A$ to $\Split$}
\begin{defn}[Intersection property]\label{intersect-good-couple}
Let $(G,F)$ be a functor from $\A\in \Pf$ to $\Split$. For any $(\alpha, a)\in \A_1$, let $\pi^{\alpha\alpha}_{\alpha a}= L^{\alpha a}_{\alpha \alpha}  R^{\alpha\alpha}_{\alpha a}$, $\pi^{\alpha\alpha}_{\alpha a}$ is a projector. For a given $\alpha$, we shall denote this collection as $\pi^{\alpha}$. \\

$(G,F)$ is said to satisfy the intersection property for $\alpha\in \A$ if $\pi^{\alpha}$ satisfies the intersection property (I') 
and is said to satisfy the intersection property if is satisfies it for any $\alpha\in \A$.

\end{defn}
\begin{proof}
As $(L,R)$ is a functor from $\A_1$ to $\Split$ (Corollary \ref{extension-left-right}) ,

$${\pi^{\alpha \alpha}_{\alpha a}}^2= L^{\alpha a}_{\alpha \alpha}R^{\alpha \alpha}_{\alpha a}L^{\alpha a}_{\alpha \alpha}R^{\alpha \alpha}_{\alpha a}= L^{\alpha a}_{\alpha \alpha}R^{\alpha \alpha}_{\alpha a}$$.
\end{proof}

\begin{rem}

Let $(G,F)$ be a functor from $\A$ to $\Split$; let $(\alpha,a)\in \A_1$, let us remark that $\pi^{\alpha \alpha}_{\alpha a}=L^{\alpha a}_{\alpha \alpha}L^{a a }_{\alpha a} F^{\alpha}_a= L^{a a}_{\alpha \alpha} F^{\alpha}_{a}= G^a_{\alpha} F^{\alpha}_{a}$. Therefore the intersection property is equivalent to for any $(\alpha,a),\in \A_1$, for any $v\in  G(\alpha)$, $ \im (G^b_{\alpha}F^{\alpha}_{b}G^a_{\alpha}F^{\alpha}_a, b\in \hat{\alpha}) \subseteq \im \zeta^{\alpha}{V_l}^{\alpha a}_{\alpha \alpha}$. If $\A$ has all its intersections, the intersection property is equivalent to for any $(\alpha,a),(\alpha,b)\in \A_1$, $G^b_{\alpha}F^{\alpha}_b G^a_{\alpha}F^{\alpha}_a= G^{b\cap a}_{\alpha}F^{\alpha}_{b\cap a} $.
\end{rem}

\begin{rem}
Let $(G,F)$ be a functor from $\A$ to $\Split$, for any $\alpha\in \A$, $(\pi^{\alpha,\alpha}_{\alpha,a}, a\in \hat{\alpha})$ is presheafable as for any $v\in R(\alpha,\alpha)$ and $(\alpha,a),(\alpha,b)\in\A$, such that $b\leq a$, $\pi^{\alpha\alpha}_{\alpha b}(v)=\pi^{\alpha a}_{\alpha b}\pi^{\alpha \alpha}_{\alpha a}(v)$. 

\end{rem}

\begin{rem}
Let $(G,F)$ be a functor from $\A$ to $\Split$, for any $(\alpha, a)\in \A_1$, $\im \pi^{\alpha\alpha}_{\alpha a}= R(\alpha,a)$.
\end{rem}

\begin{prop}
Let $(G,F)$ be a functor from $\A$ to $\Split$, if $(G,F)$ verifies the intersection property for $\alpha$ then it verifies it for any $\beta\leq \alpha$.

\end{prop}


\subsection{Equivalence between intersection property and decomposability}

\begin{prop}\label{limite-proj}
Let $(G,F)$ be a functor from $\A$ to $\Split$, for any $(\alpha,a)\in \A_1$ let us denote $R|_{(\alpha,\hat{a})}$ as $R^{\alpha,\hat{a}}$; $(R^{\alpha a}_{\alpha b},b\leq a)$ is a natural transformation from $R(\alpha,a)$ to $R^{\alpha,\hat{a}}$, let us denote $\phi(\alpha,a): R(\alpha,a)\to \underset{b}{\lim} R^{\alpha,\hat{a}}(b) $ its limit. One has that $(\underset{b}{\lim} R^{\alpha,\hat{a}}(b), (\alpha,a)\in \A_1)$ is a subobject of ${V_r}$ that we shall note as $M$ and $\phi:R\to M$ is a natural transformation, furthermore it is an isomorphism.

\end{prop}

\begin{proof}
By definition $\underset{b}{\lim}R^{\alpha,\hat{a}}(b)\subseteq V(\alpha,a )$; let $\alpha,a,b,c,c_1\in \A$ such that $\alpha\geq a\geq b\geq c\geq c_1$ and $v\in \underset{c}{\lim}R^{\alpha,\hat{a}}(c)$, one has that $R^{\alpha c}_{\alpha c_1}({V_r}^{\alpha a}_{\alpha b}(v)(c))=R^{\alpha c}_{\alpha c_1}(v_c)=v_{c_1}$. Let $\alpha\geq \beta \geq a\geq b\geq b_1$ and $v\in \underset{c}{\lim}R^{\alpha,\hat{a}}(c)$, one has that ${V_r}^{\alpha a}_{\beta a}(v)(b)\in R(\beta,b)$ and $R^{\beta b}_{\beta b_1}( {V_r}^{\alpha a}_{\beta a}(v)(b))= R^{\beta b}_{\beta b_1} F^{\alpha}_{\beta}(v_b)=R^{\beta b}_{\beta b_1}R^{\alpha b}_{\beta b}(v_b)= R^{\alpha b_1}_{\beta b_1}R^{\alpha b}_{\alpha b_1}(v_b)= R^{\alpha b_1}_{\beta b_1}(v_{b_1})=F^\alpha_\beta(v_{b_1})$ and therefore $R^{\beta b}_{\beta b_1}( {V_r}^{\alpha a}_{\beta a}(v)(b))= {V_r}^{\alpha a}_{\beta a}(v)(b_1)$. For any $(\alpha,a),(\beta,b)\in \A_1$ such that $(\alpha,a)\geq (\beta,b)$, ${V_r}^{\alpha a}_{\beta b}( \underset{c}{\lim}R^{\alpha,\hat{a}}(c))= {V_r}^{\alpha b}_{\beta b} {V_r}^{\alpha a}_{\alpha b}(\underset{c}{\lim}R^{\alpha,\hat{a}}(c))\subseteq {V_r}^{\alpha b}_{\beta b} (\underset{c}{\lim}R^{\alpha,\hat{b}}(c))\subseteq \underset{c}{\lim}R^{\beta,\hat{b}}(c))$.\\

Let $\alpha \geq a\geq b\geq c$ and $v\in R(\alpha,a)$, $M^{\alpha a}_{\alpha b}\phi(\alpha,a)(v)(c)= R^{\alpha a}_{\alpha c}(v)=R^{\alpha b}_{\alpha c}R^{\alpha a}_{\alpha b}(v)= \phi(\alpha,b)R^{\alpha a}_{\alpha b}(v)(c)$. For $\alpha \geq \beta \geq a\geq b$, $M^{\alpha a}_{\beta a}\phi(\alpha, a)(v)(b)= F^\alpha_\beta( R^{\alpha a}_{\alpha b}(v))= R^{\beta a}_{\beta b}R^{\alpha a}_{\beta a}(v)=  \phi(\alpha, a)R^{\alpha a}_{\beta a}(v)(b)$. Therefore for any $(\alpha,a),(\beta,b)\in \A_1$ such that $(\alpha,a)\geq (\beta,b)$, $M^{\alpha a}_{\beta b}\phi(\alpha,a)= \phi(\beta,b) R^{\alpha a}_{\beta b}$.\\

For any $(\alpha,a)\in \A_1$, $(\alpha,a)$ is initial in $(\alpha, \hat{a})$, therefore $\phi(\alpha,a)$ is an isomorphism. therefore so is $\phi$.

\end{proof}

\begin{prop}\label{injection-in-sum}
Let $(G,F)$ be a functor from $\A\in \Pf$ to $\Split$, if $(G,F)$ satisfies the intersection property, $j=\mu i^M_{V_r} \phi \in \Split^{\A_1}((L,R),(V_l,V_r))$ is a monomorphism.
\end{prop}
\begin{proof}
Let $(\alpha,a)\in \A_1$ and $v\in R(\alpha, a)$, for any $b\leq a$, $j(\alpha,\alpha)L^{\alpha a}_{\alpha \alpha}(v)(b)= \mu(\alpha,\alpha)(R^{\alpha \alpha}_{\alpha c}(v),c\leq \alpha) (b)= \mu(\alpha,a)(R^{\alpha a}_{\alpha c}(v), c\leq a)(b)$ as $\mu$ is a endomorphism of $V_r$. Furthermore as $R(\alpha,a)= \im \pi^{\alpha\alpha}_{\alpha a}$, $\phi(\alpha,\alpha)L^{\alpha a}_{\alpha \alpha}(v)= \phi(\alpha,\alpha)(v)\in \im \zeta^{\alpha}{V_l}^{\alpha a}_{\alpha \alpha}$, therefore for any $ b\not\leq a$, $\mu(\alpha,\alpha)(\phi(\alpha,\alpha)(v))(b)=0$ and so $j(\alpha,\alpha)L^{\alpha a}_{\alpha \alpha}= {V_l}^{\alpha a}_{\alpha \alpha}j(\alpha, a)$. Therefore $j(\alpha,\alpha)L^{\alpha a}_{\alpha \alpha}= {V_l}^{\alpha a}_{\alpha \alpha}j(\alpha a)$.\\

Let $b\leq a$, ${V_l}^{\alpha a}_{\alpha\alpha}j(\alpha,a)L^{\alpha b}_{\alpha a}= j(\alpha,\alpha)L^{\alpha b}_{\alpha \alpha}= {V_l}^{\alpha b}_{\alpha \alpha}j(\alpha,b)= {V_l}^{\alpha a}_{\alpha\alpha}{V_l}^{\alpha b}_{\alpha a}j(\alpha, b)$. ${V_l}^{\alpha a}_{\alpha \alpha}$ is a monomorphism therefore, $j(\alpha,a)L^{\alpha b}_{\alpha a}={V_l}^{\alpha b}_{\alpha a}j(\alpha b)$. \\

For $\alpha\geq \beta \geq a$, $j(\alpha,a)L^{\beta a}_{\alpha a}={V_l}^{\beta a}_{\alpha a} j(\beta,a)$ holds even if $(G,F)$ does not satisfy the intersection property.\\

$\phi$ is a isomorphism from $R$ to $M$, $i^M_{V_r}$ is a monomorphism and $\mu$ is an isomorphism from $V_r$ to $V_r$ therefore $j$ is a monomorphism.
\end{proof}

\begin{prop}
Let $(G,F)$ be a functor from $\A$ to $\Split$, let $b\in \A$. For any $(\alpha,a)\in \A_1$ if $a\geq b$ let $G_{b}(\alpha, a)= G(\alpha)$ and if not $G_{b}(\alpha,a)=0$; let $\alpha\geq a\geq a_1$, if $a_1\geq b$, let ${G_b}^{\alpha a}_{\alpha a_1}= {G^b}^{\alpha a_1}_{\alpha a}=\id$, if not ${G_b}^{\alpha a}_{\alpha a_1}={G^b}^{\alpha a_1}_{\alpha a}=0$, let $\alpha\geq \beta\geq a$, if $a\geq b$ let ${G^b}^{\alpha a}_{\beta a}= F^\alpha_\beta$ and ${G_b}^{\beta a}_{\alpha a}= G^\beta_\alpha$,  if not, ${G_b}^{\alpha a}_{\beta a}=0= {G^b}^{\beta a}_{\alpha a}$. $(G^b,G_b)$ extend into a unique functor from $\A_1$ to $\Split$. For any $(\alpha,a)\in \A_1$, if $a\geq b$ let $pr_b(\alpha,a):V(\alpha,a)\to G(\alpha)$ be such that for any $v\in V(\alpha,a)$, $pr_b(v)=v_b$, if $a\not\geq b$, $pr_b(v)=0$. $\pr_b\in \Split^{\A_1}((V_l,V_r), (G_b,G^b))$.

\end{prop}
\begin{proof}
Let $b\in \A$, $\alpha \in \A$, $G_b|_{(\alpha, \hat{\alpha})}= G(\alpha) 1[(\alpha,b)\leq .]$ (if $b\not\leq \alpha)$, $1[(\alpha,b)\leq .]=0$ as the relation is taken in $\A_1$), and $G^b|_{(\alpha, \hat{\alpha})^{op}}= G(\alpha) 1[a\leq .]$, therefore for $\alpha \geq a\geq a_1\geq a_2$, ${G_b}^{\alpha a_1}_{\alpha a}{G_b}^{\alpha a_2}_{\alpha a_1}={G_b}^{\alpha a_2}_{\alpha a}$, ${G^b}^{\alpha a_1}_{\alpha a_2}{G^b}^{\alpha a}_{\alpha a_1}={G^b}^{\alpha a}_{\alpha a_2}$.\\

Let $\alpha\geq \beta\geq \gamma\geq a$, if $a\geq b$, ${G_b}^{\beta a}_{\alpha a} {G_b}^{\gamma a}_{\beta a}= G^{\beta}_{\alpha}G^{\gamma}_{\beta}= {G_b}^{\gamma a}_{\alpha a}$, ${G^b}^{\beta a}_{\gamma a}{G^b}^{\alpha a}_{\beta a}= F^\beta_\gamma F^\alpha_\beta= {G^b}^{\alpha a}_{\gamma a}$, and if $a\not\geq b$, ${G_b}^{\beta a}_{\alpha a} {G_b}^{\gamma a}_{\beta a}=0= {G_b}^{\gamma a}_{\alpha a}$, ${G^b}^{\beta a}_{\gamma a}{G^b}^{\alpha a}_{\beta a}=0={G^b}^{\alpha a}_{\gamma a}$.\\

Furthermore for $\alpha \geq \beta\geq a\geq c$, if $c\geq b$, ${G_{b}}^{\beta a}_{\alpha a}{G_b}^{\beta c}_{\beta a}= G^{\beta}_{\alpha}= {G_b}^{\alpha c}_{\alpha a}{G_b}^{\beta c}_{\alpha b}$, ${G^b}^{\beta a}_{\beta c}{G^b}^{\alpha a}_{\beta a}= F^\alpha_\beta ={G^b}^{\alpha c}_{\beta c}{G^b}^{\alpha a}_{\alpha c}$. Therefore by Proposition \ref{extension-functor-poset}, $G_a$, $G^a$ extend respectively to a functor and a presheaf from $\A_1$ to $\Mod$.\\

Let $\alpha \geq a \geq a_1$, if $b\leq a_1$, ${G^b}^{\alpha a}_{\alpha a_1}{G_b}^{\alpha a_1}_{\alpha a }=\id$ if $b\not\leq a_1$, $ {G^b}^{\alpha a}_{\alpha a_1}{G_b}^{\alpha a_1}_{\alpha a }: 0\to 0$ therefore ${G^b}^{\alpha a}_{\alpha a_1}{G_b}^{\alpha a_1}_{\alpha a }=\id$; let $\alpha\geq \beta\geq a$, ${G^b}^{\alpha ,a}_{\beta a, }{G_b}^{\beta,a }_{\alpha,a}=\id$; therefore $(G_b,G^b)$ is a functor from $\A_1$ to $\Split$.\\

Let $\alpha \geq a\geq a_1$, if $b\leq a_1$, $pr_b(\alpha, a){V_l}^{\alpha a_1}_{\alpha a}= pr_b(\alpha, a_1)= {G_b}^{\alpha a_1}_{\alpha a} pr_b(\alpha, a_1)$, $pr_b(\alpha, a_1) {V_r}^{\alpha a}_{\alpha a_1}=pr_b(\alpha, a)= {G^b}^{\alpha a}_{\alpha a_1}pr_b(\alpha, a)$, and if $b\not\leq a_1$, for any $v\in V_l(\alpha, a_1)$, ${G_b}^{\alpha a_1}_{\alpha a} pr_b(\alpha, a_1)(v)=0=v_b1[b\leq a_1]=pr_b(\alpha, a){V_l}^{\alpha a_1}_{\alpha a} $ and $pr_b(\alpha, a_1){V_r}^{\alpha a}_{\alpha a_1}=0={G_b}^{\alpha a}_{\alpha a_1}pr_b(\alpha, a)$. For $\alpha \geq \beta \geq a$, if $b\leq a$, $pr_b(\alpha,a){V_l}^{\beta a}_{\alpha a}= G^\beta_\alpha pr_b(\beta,a)={G_b}^{\beta a}_{\alpha a}pr_b(\beta, a)$ and $pr_b(\beta,a){V_r}^{\alpha a}_{\beta a}=F^\alpha_\beta pr_b(\alpha,a)={G^b}^{\alpha a}_{\beta a}pr_b(\alpha,a)$; if $b\not\leq a$, $pr_b(\beta,a){V_r}^{\alpha a}_{\beta a}=0= {G^b}^{\alpha a}_{\beta a}pr_b(\alpha,a)$ and $pr_b(\alpha ,a){V_l}^{\beta a}_{\alpha a}=0= {G_b}^{\beta a}_{\alpha a} pr_b(\beta,a)$.

\end{proof}

\begin{defn}
Let $(G,F)$ be a functor from $\A\in \Pf$ to $\Split$. For any $a\in \A$ let $(S_a,S^a)=\im pr_a\circ j:\A_1 \to \Split$
\end{defn}

\begin{prop}\label{thm-presheaf}
Let $(G,F)$ a functor from $\A\in \Pf$ to $\Split$ that satisfies the intersection property, then $j|^{\im j}(L,R)\to (\underset{a\in \A_1}{\bigoplus}S_a,\underset{a\in \A_1}{\bigoplus}S^a)$ is an isomorphism. We shall note $j|^{\im j}$ as $\psi$.  
\end{prop}
\begin{proof}
From Proposition \ref{injection-in-sum} $j$ is a monomorphism furthermore from Corollary \ref{isomorphism-decomposable-intersection} for any $(\alpha,a)\in \A_1$, $\im j(\alpha,a)=\underset{b\leq a}{\bigoplus}S_b(\alpha,a)=\underset{b\in \A}{\bigoplus}S_b(\alpha,a)$.
\end{proof}

\begin{thm}\label{decomposable-intersection-presheaf}
Let $(G,F)$ be a functor from $\A\in \Pf$ to $\Split$, $(G,F)$ satisfies the intersection property if and only if $(G,F)$ is decomposable. 

\end{thm}

\begin{proof}
Let $(G,F)$ be a functor from $\A\in \Pf$ to $\Split$ that satisfies the intersection property; let us recall that $(L,R)$ is a functor from $\A_1$ to $\Split$ (Corollary \ref{extension-left-right}) and that for any $\alpha\geq \beta\geq a$, $L^{\beta a}_{\alpha a}$ is an isomorphism. Let us recall that for any $\alpha\geq \beta \geq a$, $\underset{b\in \A}{\bigoplus}{S_b}^{\beta a}_{\alpha a}\psi(\beta,a)=\psi(\alpha,a)L^{\beta a}_{\alpha a}$, where $\psi(\alpha,a)$, $\psi(\beta,a)$ (Proposition \ref{thm-presheaf}), $L^{\alpha a}_{\alpha a}$ are isomorphisms; therefore $ {\underset{b\in \A}{\bigoplus}S_b}^{\beta a}_{\alpha a}$ is an isomorphism and as $(\underset{b\in \A}{\bigoplus}S_b,\underset{b\in \A}{\bigoplus}S^b)$ is a functor from $\A_1$ to $\Split$ its inverse is $\underset{b\in \A}{\bigoplus}{S^b}^{\alpha a}_{\beta a}$ and ${S_a}^{\beta a}_{\alpha a}$ is an isomorphism (its inverse is ${S^a}^{\alpha a}_{\beta a}$); for $\alpha\geq a\geq a_1\geq b$, $ {S_a}^{\alpha a_1}_{\alpha a}$ is also an isomorphism.\\

Let us remark that $\{(a,a)\in A_1: a\in\A\}$ is isomorphic to $\A$; let for any $a\in \A$, $C^a=S^a|_{\A}$ and $C_a=S_a|_\A$; by definition $(C^a,C_a)$ is a functor from $\A$ to $\Split$; for any $a_2\geq a_1 \geq a$, ${C_a}^{a_1}_{a_2}= {S_a}^{a_1 a_1}_{a_2,a_2}= {S_a}^{a_2 a_1}_{a_2 a_2} {S_a}^{a_1 a_1}_{a_2 a_1}= {S_a}^{a_1 a_1}_{a_2 a_1}$ is an isomorphism. Furthermore $(G,F)\cong (L,R)|_\A$ (Remark \ref{morphism-functor-to-extension}), therefore $(G,F)$ is decomposable and its decomposition is $(\underset{a\in \A}{\bigoplus}C^a,\underset{a\in\A}{\bigoplus}C_a)$.\\

Let $\A\in \Pf$, $(G,F)=(\underset{a\in \A}{\bigoplus}C^a,\underset{a\in\A}{\bigoplus}C_a)$, let $\alpha \in \A$, for any $a\leq \alpha$ and $c\leq \alpha$ and $v\in G(\alpha)$, ${\pi^{\alpha}}_a(v)(c)= G^a_\alpha F^\alpha_a(v) (c)= v_c1[c\leq a]$. Let us denote $\pi^\alpha$ simply as $\pi$ and $\mu_{\hat{\alpha}}$ as $\mu$; for any $v\in \im \Pi(\pi)$ and $a\in \hat{\alpha}$, $\mu(v)(a)=\underset{b\leq a}{\sum}\mu(a,b) \underset{c\leq b}{\bigoplus} v_c= \underset{c\leq a}{\bigoplus}\underset{\substack{b :\\ c\leq b\leq a} }{\sum} \mu(a,b)v_c= v_a\times{a}$. 

Furthermore for $a\leq \alpha$ and $b\leq \alpha$ and $c\leq \alpha$ and $v\in G(\alpha)$, $\pi_a\pi_b(v)(c)=v_c1[c\leq a \& c\leq b]$, $\pi_a\pi_b(v)= \underset{\substack{c\leq a \\ c\leq b}}{\sum} v_c\times c=\underset{\substack{c\leq a \\ c\leq b}}{\sum}\mu(v)(c)$. Therefore $(G,F)$ satisfies the intersection property.

\end{proof}

\begin{shownto}{long}

\section{Concluding remarks}

\wk ajouter la decomposition canonique en parlant des $S_a$ pour le mu bien defini et parler d'une autre caracterisation du produit pour lequel on peut avoir une autre preuve, peut etre donner cette autre preuve car je l'ai vu donc autant l'ajouter

\end{shownto}

\bibliographystyle{ieeetr}
\bibliography{bintro}

\begin{thebibliography}{10}

\bibitem{GS2}
G.~Sergeant-Perthuis, ``Intersection property and interaction decomposition.''
  arXiv:1904.09017v2, 2019.

\bibitem{Kellerer1964}
H.~G. Kellerer, ``Masstheoretische {M}arginalprobleme,'' {\em Mathematische
  Annalen}, vol.~153, no.~3, pp.~168--198, 1964.

\bibitem{Lauritzen}
S.~L. Lauritzen, {\em Graphical Models}.
\newblock Oxford Science Publications, 1996.

\bibitem{Speed}
T.~P. Speed, ``A note on nearest-neighbour gibbs and markov probabilities,''
  {\em Sankhy\=a: The Indian Journal of Statistics, Series A}, 1979.

\bibitem{Yeung}
T.~H. Chan and R.~W. Yeung, ``Probabilistic inference using function
  factorization and divergence minimization,'' in {\em Towards an Information
  Theory of Complex Networks: Statistical Methods and Applications} (M.~D.
  et~al., ed.), ch.~3, pp.~47--74, Springer, 2011.

\bibitem{Matus}
F.~Mat{\'u}{\v{s}}, ``Discrete marginal problem for complex measures,'' {\em
  Kybernetika}, vol.~24, no.~1, pp.~36--46, 1988.

\bibitem{Davidson}
R.~Davidson, ``Determination of confounding,'' in {\em Stochastic Analysis}
  (D.~Kendall and E.~Harding, eds.), John Wiley and Sons, 1973.

\bibitem{Haberman}
S.~J. Haberman, ``Direct products and linear models for complete factorial
  tables,'' {\em Annals of Statistics}, 1975.

\bibitem{Sinai}
Y.~G.~S. Sinai, {\em Theory of phase transitions: Rigourous results}.
\newblock Pergamon Press, 1982.

\bibitem{GS1}
G.~Sergeant-Perthuis, ``Bayesian/graphoid intersection property for
  factorisation models.'' arXiv:1903.06026v2, 2019.

\bibitem{Rota}
G.-C. Rota, ``On the foundations of combinatorial theory {I}. {Theory} of
  {M}{\"o}bius functions,'' {\em Probability theory and related fields},
  vol.~2, no.~4, pp.~340--368, 1964.

\bibitem{coend}
F.~Loregian, ``Coend calculus.'' arXiv:1501.02503, 2019.

\end{thebibliography}





\end{document}